\documentclass{amsart}
\usepackage{amsfonts,amsthm,amsmath}
\usepackage[margin=1in]{geometry}
\usepackage{graphicx}

\newcommand{\BB}{{\mathbb{B}}}
\newcommand{\RR}{{\mathbb{R}}}

\newcommand{\rp}{{\rho^\prime}}
\newcommand{\rpp}{{\rho^{\prime\prime}}}
\newcommand{\delr}{{\Delta_r\rho}}
\newcommand{\rstar}{{r^*}}
\newcommand{\gstar}{{\gamma^*}}
\newcommand{\Ostar}{{\Omega^*}}

\newcommand{\jpp}{{j_1^{\prime\prime}}}
\newcommand{\ipp}{{i_1^{\prime\prime}}}

\newcommand{\xmax}{{x_{\max}}}
\newcommand{\xmid}{{x_{\text{mid}}}}
\newcommand{\bmax}{{b_{\max}}}
\newcommand{\bmaxsq}{{b_{\max}^2}}

\newcommand{\bmidsq}{{b_{\text{mid}}^2}}
\newcommand{\tmax}{{\tau_{\max}}}
\newcommand{\tmin}{{\tau_{\min}}}
\newcommand{\tmid}{{\tau_{\text{mid}}}}

\newcommand{\thetahat}{{\hat{\theta}}}

\newcommand{\Yl}{{Y_l}}

\newcommand{\pll}{{p_{1,1}}}
\newcommand{\pllsq}{{p_{1,1}^2}}
\newcommand{\sproj}{{P_{\partial\Omega}}}

\newcommand{\sdiv}{{\mydiv_{\partial\Omega}}}

\DeclareMathOperator{\BigO}{O}

\DeclareMathOperator{\mydiv}{div}

\newtheorem{thm}{Theorem}[section]
 \newtheorem{cor}[thm]{Corollary}
 
 \newtheorem{prop}[thm]{Proposition}
\newtheorem{fact}{Fact}
\newtheorem{lemma}[thm]{Lemma}
\newtheorem*{conj}{Conjecture}

\theoremstyle{remark}
\newtheorem*{rmk}{Remark}

\begin{document}
\title[Free plate isoperimetric inequality]{An isoperimetric inequality for fundamental tones of free plates with nonzero Poisson's ratio}
\author[L. M. Chasman]{L. M. Chasman}
\address{%
Division of Math and Science\\ University of Minnesota - Morris\\ 600 E. 4th Street\\Morris
MN 56267\\ U.S.A.} 
\email{chasmanm@morris.umn.edu}

\begin{abstract} We establish a partial generalization of a prior isoperimetric inequality for the fundamental tone (first nonzero eigenvalue) of the free plate to plates of nonzero Poisson's ratio. 

Given a tension $\tau>0$ and a Poisson's ratio $\sigma$, the free plate eigenvalues $\omega$ and eigenfunctions $u$ are determined by the equation $\Delta\Delta u-\tau\Delta u=\omega u$ together with certain natural boundary conditions which involve both $\tau$ and $\sigma$. The boundary conditions are complicated but arise naturally from the plate Rayleigh quotient, which contains a Hessian squared term $|D^2u|^2$. We prove the free plate isoperimetric inequality in the $\sigma=0$ case holds for certain nonzero $\sigma$ and positive $\tau$ in the case where the fundamental mode is assumed to have simple angular dependence. We conjecture that the inequality holds for all dimensions, $\tau>0$, and relevant values of $\sigma$, and discuss numerical and analytic support of this conjecture.

As in the case of $\sigma=0$, we adapt Weinberger’s method from the corresponding free membrane
problem, taking the fundamental modes of the unit ball as trial functions. These solutions are a linear combination of Bessel and modified Bessel functions.
\end{abstract}

\subjclass{Primary 35P15. Secondary 35J40, 35J35}
\keywords{isoperimetric, free plate, bi-Laplace}

\maketitle

\section{Introduction} 
The eigenvalues of the Laplacian operator and its fourth-order cousin the bi-Laplacian appear in many models of physical situations, representing quantities such as frequency or energy. One classic example is that the eigenvalues $\mu$ of the Neumann Laplacian on a bounded region $\Omega$ determine the frequencies of vibration of a free membrane with that shape. If $\Omega^*$ is the ball of same volume as $\Omega$, then we have
\[
\mu_1(\Omega) \leq \mu_1(\Ostar)\qquad\text{with equality if and only if $\Omega$ is a ball.}
\]
First conjectured by Kornhauser and Stakgold \cite{KS52}, this isoperimetric inequality was proved for simply connected domains in $\RR^2$ by Szeg\H o \cite{S50,Serrata} and extended to all domains and dimensions by Weinberger \cite{W56}. 

While Laplacian eigenvalue problems represent vibrations of membranes, there are corresponding bi-Lapalace problems represent vibrations and buckling energies of plates. Fourth-order plate problems are frequently more difficult than their second-order analogs -- the theory of the bi-Laplace operator is not as well understood, and because the order is higher, exact solutions (such as those used for trial function methods) can require more complicated linear combinations of special functions. 

The isopermetric inequality for the fundamental tone of the Dirichlet Laplacian (drum) was proved by Faber \cite{Faber23} and Krahn \cite{krahn25,krahn26} in the 1920's with the ball as the minimizer. It was not until the 1980's and 90s' that it was proved that the ball also the lower bound for the first clamped plate eigenvalue \cite{T81, N92, N95, AB95plate}. The methods used by Talenti, Nadirashvilli, Ashbaugh and Benguria to prove the clamped plate isoperimetric inequality are quite different than those for the free plate and membrane and only establish the bound in dimensions 2 and 3. The problem remains open for dimensions four and higher, with a partial result by Ashbaugh and Laugesen \cite{AL96}. 

For forth-order problems. other boundary conditions exist, such as the hinged plate investigated by Nazarov and Sweers \cite{NS07}. Other generalizations of the Sz\"ego-Weinberger inequality include its analog in other spaces. In spaces of constant curvature, the spherical cap (analog of the ball) maximizes the first Neumann eigenvalue, as seen by Ashbaugh and Benguria\cite{AB95} In Gauss space, the problem was cconsidered by Chiacchio and Di Blasio \cite{CdB12}. As in Euclidean space this gives an upper bound on the fundamental tone of the Neumann Gaussian Laplacian (Hermitian); one can also consider lower bounds on the Neumann eigenvalues, eg, \cite{BCHT13,BCT13}.

In \cite{inequalitypaper}, we adapted Weinberger's trial function argument to prove the free plate analog of the Sz\"ego-Weinberger inequality for plates with positive tension and assuming the Poisson's ratio (a property of the material) of the plate was zero. Taking $\omega_1$ to be the fundamental tone, we had that:
\begin{equation}\label{plateiso}
\omega_1(\Omega) \leq \omega_1(\Ostar)\qquad\text{with equality if and only if $\Omega$ is a ball.}
\end{equation}
In this paper we present a generalization of this result to some plates under tension with \emph{nonzero} values of Poisson's ratio. We prove that if the dimension, tension, and Poisson's ratio are such that fundamental mode of the ball has simple angular dependence, we again have the bound~\eqref{plateiso}. Our proof relies on the variational characterization of eigenvalues with suitable trial functions, taken to be extensions of the fundamental mode of the unit ball. This follows both Weinberger's approach for the free membrane and our prior work for the free plate with zero Poisson's ratio in \cite{inequalitypaper}.  However, because the plate equation is fourth order, finding the trial functions and establishing the appropriate monotonicities is significantly more complicated than in the membrane case. The inclusion of $\sigma$ further complicates matters and prevents us from applying some of our tools from \cite{inequalitypaper}, including identifying the fundamental mode of the ball.

Based on numerical evidence and analytic reasoning, we conjecture that the fundamental mode of the ball has simple angular dependence (ie, in dimension 2 the angular part can be written as $\sin(\theta)$ or $\cos(\theta)$) for all dimensions, positive tension and mathematically suitable values of Poisson's ratio, and so the isoperimetric inequality~\eqref{plateiso} holds for all plates.

Poisson's ratio, which we will denote by $\sigma$, is a property of the material of the plate. If a material is stretched in one direction, it usually contracts in the orthogonal directions; the value $\sigma$ is a ratio of the strains. In some materials the material expands in the orthogonal directions rather than contracting; these materials have $\sigma<0$ and are called auxetic.

Considering nonzero Poisson's ratio is a natural generalization of the free plate problem because $\sigma$ appears in the Rayleigh quotient for the plate, even though it does not appear in the eigenvalue equation itself; instead, we see explicit dependence on $\sigma$ in the natural boundary conditions. Verchota recently established the solvability of the biharmonic Neumann problem \cite{verchota}, the boundary conditions for which arise from the zero-tension plate with nonzero Poisson's ratio. Interestingly, the clamped plate problem is independent of $\sigma$. Although the clamped plate problem begins from the plate Rayleigh quotient, integration by parts and the imposed boundary conditions allow the clamped plate quotient to be written in its more familiar form, which is independent of $\sigma$.

This paper proceeds as follows: we begin by formulating the problem and stating our main theorem, a partial result towards the conjectured isoperimetric inequality. We then prove existence of the discrete spectrum and regularity of the eigenfunctions for specific values of $\sigma$ in Section~\ref{specsec} and use the Rayleigh quotient to establish bounds on the fundamental tone as a function of $\sigma$ and $\tau$ in Section~\ref{functausec}. 

To prepare to prove the theorem, we establish crucial properties of ultraspherical Bessel functions in Section~\ref{Besselsec}, derive the form of the natural boundary conditions in Section~\ref{bcsec}. We use these in Section~\ref{fundtonesec} to find the eigenfunctions of the ball, where we also state and discuss our conjecture that the fundamental mode has simple angular dependence. We use the fundamental mode to construct our trial functions and establish some properties of these in Section~\ref{trialfcnsec}. From there we proceed to prove our main theorem in Sections~\ref{monotonesec} and~\ref{proofsec}.

\section{ Formulating the problem} 
Let $\Omega\subseteq\RR^d$ be a smoothly bounded region. We will write $\Omega^*$ for the ball in $\RR^d$ with the same volume as $\Omega$.

The generalized plate Rayleigh quotient has the form
\begin{equation}\label{RQ}
 Q_\Omega[u]:=\frac{\int_\Omega (1-\sigma)|D^2u|^2+\sigma(\Delta u)^2+\tau|Du|^2\,dx}{\int_\Omega u^2\,dx}.
\end{equation}
Here the parameter $\tau$ measures tension over flexural rigidity, and $\sigma$ is Poisson's ratio. A positive $\tau$ represents a plate under tension; taking $\tau<0$ gives us a plate under compression. The limiting case as $\tau\to\infty$ is more naturally understood as the limit as \emph{rigidity} goes to zero; in other words, the plate should behave like a membrane for larger $\tau$. For Poisson's ratio, typically $\sigma\in[0,0.5]$ for real-world materials, although a class of materials known as auxetics have negative Poisson's ratio. We will take $\sigma$ to be in $(-1/(d-1),1)$ in order to be assured of coercivity of our form.

From the Rayleigh quotient \eqref{RQ}, we derive the partial differential equation and boundary conditions governing the vibrational modes of a free plate. The critical points of the quotient are the eigenstates for the plate satisfying the free boundary conditions and the critical values are the corresponding eigenvalues. We shall show in Section~\ref{bcsec} that the differential eigenvalue equation is
\begin{equation}
\Delta \Delta u - \tau \Delta u = \omega u, \label{mainineq}
\end{equation}
where $\omega$ is the eigenvalue, with the natural (\emph{i.e.}, unconstrained or ``free'') boundary conditions on $\partial\Omega$:
\begin{align}
&Mu := \frac{\partial^2 u}{\partial n^2}= 0 \label{BC1}\\
&Vu := \tau\frac{\partial u}{\partial n}-\sdiv\left(\sproj\left[(D^2u)n\right]\right)-\frac{\partial(\Delta u)}{\partial n} = 0\label{BC2}.
 \end{align}
 Here $\partial/\partial n$ denotes the normal derivative and $\sdiv$ is the surface divergence, and  $\sproj$ projects a vector into the tangent space of $\partial\Omega$. 
 
The fundamental tone (lowest nonzero eigenvalue) of the plate with shape $\Omega$ can then be written with the Rayleigh-Ritz characterization as follows:
\[
 \omega_1(\Omega)=\inf\left\{Q_\Omega[u]:u\in H^2{\Omega}, \int_\Omega u\,dx=0.\right\}
\]
We conjecture the following isoperimetric inequality:
\begin{conj}\label{mainconj} Let $\Omega\subset\RR^d$ be a smoothly bounded region, and supposed we have $\tau>0$ and $\sigma\in(-1/(d-1),1)$ fixed. Then
\[
\omega_1(\Omega)\leq\omega_1(\Ostar),                                                     
\]
with equality if and only if $\Omega=\Ostar$.
\end{conj}
This conjecture has previously been proved true in the case of $\sigma=0$ in \cite{inequalitypaper}. It is supported by numerical evidence and analytic arguments made in Section~\ref{fundtonesec} and our weaker result.

\section{ Main Result}
In this paper we will prove the following result:
\begin{thm}\label{mainthm} Supposed we have $\tau>0$ and $\sigma\in(-1/(d-1),1)$ fixed so that the fundamental mode of the ball $\Ostar$ has simple angular dependence. Suppose also that one of the following hold:
\begin{itemize}
 \item $d=2$ and $\sigma>-51/97$ or $\tau\geq 3(1-\sigma)/(1+\sigma)$,
 \item $d=3$,
 \item $d\geq 4$ and $\sigma\leq0$ or $\tau\geq (d+2)/2$.
\end{itemize}
Then
\begin{equation}\label{maineq}
\omega_1(\Omega)\leq\omega_1(\Ostar),                                                     
\end{equation}
with equality if and only if $\Omega=\Ostar$.
\end{thm}
The restrictions $\sigma>-51/97$ for $d=2$ and the lower bounds on $\tau$ are specific to our method of proof and do not seem to be inherent to the problem. Furthermore, numerical and analytic evidence suggest that the fundamental mode of the ball $\Ostar$ has simple angular dependence for choices of dimension $d\geq 2$, Poisson's ratio $\sigma\in(-1/(d-1),1)$ and tension $\tau>0$. We will argue this more thoroughly in Section~\ref{fundtonesec}.

The proof of Theorem~\ref{mainthm} is a trial function argument like that of Weinberger \cite{W56} and our own proof of the the $\sigma=0$ case in \cite{inequalitypaper}. It proceeds from a sequence of lemmas, organized into the following sections:
\begin{itemize}
  \item Section~\ref{trialfcnsec} Define the trial functions and prove crucial properties about concavity of the radial part.
  \item Section~\ref{monotonesec} Prove partial monotonicity of the Rayleigh quotient, treating the cases of positive and negative $\sigma$ separately.
  \item Section~\ref{proofsec} Complete the proof using the partial monotonicity and rescaling and rearrangement arguments.
\end{itemize}

\section{ The existence of the spectrum}\label{specsec}
The weak eigenvalue problem is given by the sesquilinear form
\begin{align*}
a(u,v) &=\int_\Omega(1-\sigma)\sum_{i,j=1}^d\overline{u_{x_ix_j}}v_{x_ix_j}+\sigma(\overline{\Delta u}\Delta v)+\tau(\overline{D u}\cdot D v)\,dx
\end{align*}
with form domain $H^2(\Omega)$. Note the plate Rayleigh quotient $Q$ can be written in terms of $a(\cdot,\cdot)$, with $Q[u]=a(u,u)/\|u\|_{L^2}^2$.

\begin{prop} \label{spect} Fix $\tau\in\RR$ and $\sigma\in\left(-\frac{1}{d-1},1\right)$. Then the spectrum of the operator $\Delta^2-\tau\Delta$ associated with the form $a(\cdot,\cdot)$ consists entirely of isolated eigenvalues of finite multiplicity
\begin{equation}\label{eigenvalueineq}
 \omega_1\leq\omega_2\leq \dots\leq\omega_n\rightarrow\infty\quad\text{as}\quad n\rightarrow\infty.
\end{equation}
Furthermore, there exists an associated set of weak eigenfunctions which is an orthonormal basis for $L^2(\Omega)$.
\end{prop}

Because the quadratic form involves a convex combination of second-order terms $|D^2u|^2$ and $|\Delta u|^2$, we will find the following inequality useful in proving Proposition~\ref{spect}: 
\begin{fact}\label{laphesbound} For any function $u\in H^2(\Omega)$, we have the sharp bound
$(\Delta u)^2\leq d|D^2u|^2$.
\end{fact}
\begin{proof}
By applying Cauchy-Schwartz, we see
\begin{align*} 
 \left(\sum_{i=1}^d u_{x_ix_i}\right)^2&\leq \left(\sqrt{d}\left(\sum_{i=1}^d u_{x_ix_i}^2\right)^{1/2}\right)^2=d\sum_{i=1}^du_{x_ix_i}^2 \leq d\sum_{i,j=1}^du_{x_ix_j}^2.
\end{align*}
 Sharpness follows from taking $u=x_1^2+x_2^2+\dots+x_d^2$.
\end{proof}

We are now ready to prove the existence of the spectrum of our Rayleigh quotient.

\begin{proof}(Proposition~\ref{spect}) We will prove that the quadratic form $a(\cdot,\cdot)$ is bounded and coercive; that is, we will show the existence of positive constants $c_1$ and $c_2$ such that
\[
a(u,u)+c_1\|u\|^2\geq c_2\|u\|^2_{H^2(\Omega)}.
\]
Once we have this, then by a standard result (see e.g., Corollary 7.7 \cite[p. 88]{SH77}), the form $a(\cdot,\cdot)$ has a set of weak eigenfunctions which is an orthonormal basis for $L^2(\Omega)$, and the corresponding eigenvalues are of finite multiplicity and satisfy \eqref{eigenvalueineq}.

To prove boundedness of the form when $\sigma\geq0$, notice that by Fact~\ref{laphesbound}, we have $(\Delta u)^2\leq d|D^2u|^2$; thus 
\[
 a(u,u)\leq \int_\Omega (1-\sigma+d\sigma)|D^2u|^2+\tau|Du|^2\,dx,
 \]
and so $a(u,u)\leq \max(1+(d-1)\sigma,\tau)\|u\|_{H^2(\Omega)}^2$. That is, $a(\cdot,\cdot)$ is bounded.

When $\sigma<0$, we note that $\sigma(\Delta u)^2\leq 0$ and so 
\[
 a(u,u)\leq \int_\Omega (1-\sigma)|D^2u|^2+\tau|Du|^2\,dx\leq \max(1-\sigma,\tau)\|u\|_{H^2(\Omega)}^2.
 \]

To establish coercivity, it is enough to show our form $a(\cdot,\cdot)$ is bounded below by a coercive quadratic form, in our case by a positive constant multiple of the quadratic form for the free plate when $\sigma=0$. This form was proved to be coercive for all $\tau$ in \cite[Prop. 2]{inequalitypaper}.

For $0\leq\sigma<1$, note that
\begin{align*}
 a(u,u) &\geq (1-\sigma)\|D^2u\|^2+\tau\|Du\|^2\\
 &=(1-\sigma)\left(\|D^2u\|^2+\frac{\tau}{1-\sigma}\|Du\|^2\right).
\end{align*}
The lower bound is $(1-\sigma)$ times the quadratic form associated with the free plate with zero Poisson's ratio and positive tension $\tau/(1-\sigma)$. Since we assumed $\sigma<1$, this establishes coercivity of $a(\cdot,\cdot)$ for $\sigma\in[0,1)$.

When instead we have a negative Poisson's ration, in particular\\ $0>\sigma>-1/(d-1)$, we use Fact~\ref{laphesbound} to obtain:
\begin{align*}
 a(u,u)&\geq (1-\sigma)\|D^2u\|^2+d\sigma\|D^2u\|^2+\tau\|Du\|^2 \\
&\geq (1+(d-1)\sigma)\|D^2u\|^2+\tau\|Du\|^2.
\end{align*}
Again, this is a constant multiple of the quadratic form associated with a free plate under tension and with zero Poisson's ratio. Because we assumed $\sigma>-1/(d-1)$, this constant is positive and hence the form $a(\cdot,\cdot)$ is coercive.
\end{proof}


\begin{prop} \label{regprop} For any $\tau \in\RR$ and smoothly bounded $\Omega$, the weak eigenfunctions associated with our form $a(\cdot,\cdot)$ are real-valued and smooth on $\overline{\Omega}$. 
\end{prop}

\begin{proof} 
Let $u$ be a weak eigenfunction of $a(\cdot,\cdot)$ with associated eigenvalue $\omega$; by Proposition~\ref{spect} we have $u\in H^2(\Omega)$. Then by a theorem in \cite[p 668]{Nir55}, we have $u\in H^k(\Omega)$ for every positive integer $k$. Thus we have $u\in H^k(\Omega)$ for all $k\in \mathbb{Z}^+$, and so $u\in C^\infty(\Omega)$.

Regularity on the boundary follows from global interior regularity and the Trace Theorem (see, for example, \cite[Prop 4.3, p. 286 and Prop 4.5, p. 287.]{taylor}). Thus we have $u\in C^\infty(\overline{\Omega})$, as desired.

Because our eigenvalues are all real-valued, if $u$ is an eigenfunction with associated eigenvalue $\omega$, we may take the complex conjugate of the eigenvalue equation and see that $\overline{u}$ is also a eigenfunction with eigenvalue $\omega$. Then the real and imaginary parts of $u$ are also eigenfunctions, and we may choose real-valued eigenfunctions for our eigenbasis.\end{proof}

\begin{rmk} It may be possible that the form $a(\cdot,\cdot)$ is coercive for values of $\sigma$ outside the given range if we impose restrictions on $\tau$, such as requiring $\tau>0$. However, note that in the case $\tau=0$ and $\sigma=1$, the form is \emph{not} coercive. In this case all $H^2(\Omega)$ harmonic functions are eigenfunctions with eigenvalue zero, and so we have an eigenvalue of infinite multiplicity. 

Furthermore, the lower bound on $\sigma$ arises from applying a sharp inequality bounding the Laplacian by the Hessian. This suggests coercivity might fail for $\sigma\leq-1/(d-1)$.
\end{rmk}

\section{ The fundamental tone as a function of $\tau$ and $\sigma$}\label{functausec}
The Rayleigh quotient depends on both $\tau$ and $\sigma$, so we can view the fundamental tone $\omega_1$ as a function in either parameter. Because $\omega_1$ is found by taking the infimum of Rayleigh quotients, and the quotients are linear in each of $\tau$ and $\sigma$, the fundamental tone is concave in each parameter. Additionally, by nonnegativity of $|Du|^2$, we have that the quotient and hence $\omega_1$ are increasing in $\tau$.

Fix $\sigma\in(-1/(d-1),1)$ and view $\omega_1$ as a function of $\tau$. Then we can prove the same linear bounds on $\omega_1(\tau)$ that were established for the $\sigma=0$ case in \cite{inequalitypaper}.

\begin{lemma}\label{wbounds} For all $\sigma\in(-1/(d-1),1)$ and $\tau>0$, we have
 \begin{equation}\label{lem31}
  \tau\mu_1 \leq\omega_1(\tau,\sigma)\leq \tau \frac{|\Omega|d}{\int_\Omega |x-\bar{x}|^2 \, dx},
 \end{equation}
where $\bar{x}=\int_\Omega x\,dx/|\Omega|$ is the center of mass of $\Omega$. In particular, when $\Omega$ is the unit ball, we have
 \begin{equation} 
  \tau\mu_1 \leq \omega_1(\tau,\sigma) \leq \tau(d+2). \label{omegabounds}
 \end{equation}
 Furthermore, the upper bounds in \eqref{lem31} and \eqref{omegabounds} hold for all $\tau\in\RR$.
\end{lemma}
\begin{proof} 
To prove the upper bound, our argument is virtually the same as that of \cite[Lemma 8]{inequalitypaper}: we use as trial functions the linear functions $u_k=x_k-\overline{x_k}$. 

\if false
For the upper bound, we consider the linear trial functions $u_k=x_k-\overline{x_k}$; then $|Du_k|=1$ and the second-order derivatives in our Rayleigh quotient vanish. We are left with 
\[
 \omega_1(\tau,\sigma)\leq Q[u_k]=\frac{\tau|\Omega|}{\int_\Omega (x_k-\overline{x_k})^2\,dx}.
\]
Clearing the denominator and summing over all $k=1,\dots,d$, we see
\[
 \omega_1(\tau,\sigma)\leq\int_\Omega|x-\overline{x}|^2\,dx\leq d\tau|\Omega|.
\]
This gives us the desired upper bound for general $\Omega$. When $\Omega$ is the unit ball, then the integral $\int_\Omega|x-\overline{x}|^2\,dx$ evaluates to $d|\Omega|/(d+2)$.
\fi

To establish the lower bound, we note that for $\sigma\geq0$, both $(1-\sigma)|D^2u|^2$ and $\sigma(\Delta u)^2$ are nonnegative and so
\[
 Q[u]\geq \frac{\int_\Omega \tau|Du|\,dx}{\int_\Omega u^2\,dx}.
\]
If $\sigma<0$, then we apply Fact~\ref{laphesbound}, and since $\sigma>-1/(d-1)$, we see
\[
 Q[u]\geq\frac{\int_\Omega (1+\sigma(d-1))|D^2u|^2+\tau|du|^2\,dx}{\int_\Omega u^2\,dx}\geq\frac{\int_\Omega \tau|Du|\,dx}{\int_\Omega u^2\,dx}.
\]
In both cases, we've bounded $Q[u]$ below by $\tau$ times the free membrane Rayleigh quotient. The lower bound on $\omega_1$ is then obtained by taking the infimum of both sides over all $u\in H^2(\Omega)$ orthogonal to a constant.
\end{proof}

\begin{lemma} \label{wbounds2} For all $\tau\in\RR$,
\[
\omega_1 \leq C(\Omega)+\tau\mu_1,
\]
where the value
\[
 C(\Omega)=\frac{\int_\Omega(1-\sigma)|D^2v|^2+\sigma(\Delta v)^2\,dx}{\int_\Omega v^2\,dx}
\]
is given explicitly in terms of the fundamental mode $v$ of the free membrane on $\Omega$.
\end{lemma}
The proof is essentially the same as in \cite[Lemma 9]{inequalitypaper} (the $\sigma=0$ case) and so is omitted.

These two lemmas give us the limiting behavior of $\omega_1$ as $\tau\to\infty$:
\begin{cor}\label{limwt} For all values of our Poisson's ratio $-1/(d-1)<\sigma<1$, we have
\[
 \frac{\omega_1(\tau,\sigma)}{\tau}\to\mu_1\qquad\text{as $\tau\to\infty$}.
\]
\end{cor}
\begin{proof}
By Lemmas~\ref{wbounds} and ~\ref{wbounds2}, we have
\[
 \mu_1\leq \frac{\omega_1(\tau,\sigma)}{\tau}\leq \mu_1+\frac{C}{\tau}.
\]
Let $\tau\to\infty$.
\end{proof}
This tells us that for sufficiently large $\tau$, we expect the free plate to behave much like the free membrane. This matches the physical interpretation of $\tau$ as the reciprocal of rigidity -- large values of $\tau$ mean less rigidity.

\section{ Properties of Bessel functions}\label{Besselsec}
In this section we will define the ultraspherical Bessel functions and summarize or prove properties that we will need to prove Theorem~\ref{mainthm}. Ultraspherical Bessel functions are the generalization of spherical Bessel functions to an arbitrary dimension $d$ and can be defined in terms of the usual Bessel functions $J_\nu(z)$ and $I_\nu(z)$. For more information on Bessel functions and their properties, see, eg, \cite{AShandbook}. 

We define the $d$-dimensional ultraspherical Bessel function of the first kind of order $l$, written $j_{l}(z)$, as follows:
\[
 j_{l}(z):=z^{-(d-2)/2}J_{l+d/2-1}(z).
\]
This function solves the $d$-dimensional ultraspherical Bessel equation
\begin{equation}\label{besseleqn}
z^2 w'' + (d-1)zw' + (z^2-l(l + d-2)) w = 0.
\end{equation}
Analogously, we define the $d$-dimensional ultraspherical modified Bessel function of the first kind of order $l$, written $i_{l}(z)$, as follows:
\begin{equation}\label{modbesseleqn}
 i_{l}(z):=z^{-(d-2)/2}I_{l+d/2-1}(z).
\end{equation}
This function solves the $d$-dimensional ultraspherical modified Bessel equation
\[
z^2 w'' + (d -1)zw'- (z^2 + l(l + d-2)) w = 0.
\]
Each of the Bessel and modified Bessel equations are second-order and have a second, linearly  independent solution; these are ultraspherical Bessel functions of the second kind $N_l(z)$ and ultraspherical modified Bessel functions of the second kind $K_l(z)$. Both of these are singular at $z=0$ of different orders.

The ultraspherical Bessel functions inherit a family of recurrence relations from their two-dimensional analogues. These are proved in \cite{besselpaper}; if the reader is satisfied with a numerical demonstration, they can be verified in Mathematica or Maple for given dimension.

\begin{lemma}\cite{besselpaper}\label{Besselprops}
 We have the following properties of ultraspherical Bessel functions.
\begin{enumerate}
 \item $\frac{d-2+2l}{z}j_l(z) = j_{l-1}(z)+j_{l+1}(z)\label{j1}$
\item $j_l^\prime(z) = \frac{l}{z}j_l(z)-j_{l+1}(z)\label{j2}$
\item $j_l^\prime(z) =j_{l-1}(z)-\frac{l+d-2}{z}j_l(z)$
\item $\frac{d-2+2l}{z}i_l(z) = i_{l-1}(z)-i_{l+1}(z)\label{i1}$
\item $i_l^\prime(z) = \frac{l}{z}i_l(z)+i_{l+1}(z)\label{i2}$
\item $i_l^\prime(z) =i_{l-1}(z)-\frac{l+d-2}{z}i_l(z)$
\item $j_l^{\prime\prime}(z) = \left(\frac{l^2-l}{z^2}-1\right)j_l(z)+\frac{d-1}{z}j_{l+1}(z)\label{j4}$
\item $i_l^{\prime\prime}(z) = \left(\frac{l^2-l}{z^2}+1\right)i_l(z)-\frac{d-1}{z}i_{l+1}(z)\label{i4}$
\end{enumerate}
\end{lemma}

We will also need a bound on the roots of the $j_l'(z)$ functions:
\begin{prop}[L. Lorch and P. Szego, \cite{LS94}]\label{propLS}
Let $p_{l,k}$ denote the $k$th positive zero of $j_l^\prime(z)$. Then for $d\geq3$ and $l\geq 1$,
\[
 \frac{l(d+2l)(d+2l+2)}{d+4l+2}<\left(p_{l,1}\right)^2<l(d+2l).
\]
In particular, for $\pll$ the first zero of $j_1^\prime$, we deduce
\[
d<(\pll)^2< d+2.
\]
This inequality holds for all $d\geq2$.
\end{prop}

We will also find the following properties of signs of Bessel functions and their derivatives to be useful:
\begin{lemma}\label{Besselsigns}\cite[Lemmas 5 through 9]{besselpaper} We have the following:
 \begin{enumerate}
  \item For $l=1,\dots,5$, we have $j_l>0$ on $(0,\pll]$.
  \item We have $j_1^\prime>0$ on $(0,\pll)$.
  \item We have $j_2^\prime>0$ on $(0,\pll]$.
  \item We have $j_1^{\prime\prime}<0$ on $(0,\pll]$.
  \item We have $j_l^{(4)}>0$ on $(0,\pll]$.
  \item The functions $j_l$ and $J_{l+d/2-1}$ have the same sign. In particular, for $l= 1,\dots,5$ and any $d\geq 2$, we have $j_l(z)>0$ for $z\leq \pll$.
 \end{enumerate}
\end{lemma}

We may also write a power series for the ultraspherical Bessel functions $j_l(z)$ and $i_l(z)$ using the series for the corresponding $J_{l+(d-2)/2}$ and $I_{l+(d-2)/2}$ functions:
\begin{align*}
 &j_l(z) = \sum_{k=0}^\infty(-1)^kc_{l,k}z^{2k+l} \qquad\text{and}
\qquad i_l(z) =\sum_{k=0}^\infty c_{l,k}z^{2k+l}\\
 &\qquad\text{where}\qquad c_{l,k}:=\frac{2^{1-d/2-2k-l}}{k!\,\Gamma(k+\frac{d}{2}+l)}.
\end{align*}
By examining the power series, it is immediate that $i_l(z)$ and its derivatives are all positive on $(0,\infty)$. Since the terms of the power series for $j_l$ and $i_l$ are the same up to a sign, we also have that the derivatives of $j_l$ are dominated by those of $i_l$:
\begin{equation}
\Big|j_l^{(m)}(z)\Big| \leq i_l^{(m)}(z)\qquad \text{for $m\geq0$, $z\geq 0$,}
\end{equation}
with equality only at $z=0$.

These power series are particularly useful, and we use them to prove some crucial bounds on the Bessel functions with which we work:
\begin{lemma}[Bessel bounds]\label{bblem} For all dimensions $d\geq 2$, we have the following bounds: 
\begin{align*}
 &c_{1,0}z-c_{1,1}z^3\leq j_1(z)\leq c_{1,0}z &\text{for all $z\in[0,\sqrt{d+2}]$,}\\
 &c_{1,0}z\leq c_{1,0}z+c_{1,1}z^3\leq i_1(z) &\text{for all $z\geq0$,}\\
\jpp(z)&\leq-d_1 z+d_2 z^3 &\text{for all $z\in[0,\sqrt{d+2}]$,}\\
\ipp(z)&\leq d_1 z+K_d(M)d_2 z^3&\text{for all $z^2\leq M$,}\\
 j_2'(z)&\geq n_0z-n_1z^3 &\text{for all $z\in[0,\sqrt{d+2}]$,}\\
 i_2'(z)&\leq n_0z+n_1K_n(M)z^3  &\text{for~~} 0\leq z^2\leq M.
\end{align*}
Here $c_{k,l}$ is the $k$th coefficient of $i_l$ as before, while $d_k$ is the $k$th coefficient of $i_1^{\prime\prime}$ and $n_k$ is the $k$th coefficient of $i_2'$, given by
 \[
  d_k=\frac{2^{1-2k-d/2}(2k+1)}{(k-1)!\Gamma(k+1+d/2)}\quad{and}\quad n_k= \frac{2^{-2k-d/2}(k+1)}{k!\Gamma(k+2+d/2)}.
 \]
The functions $K_d$ and $K_n$ are given by
\[
K_d(M)=\frac{7}{5}+\frac{8}{5M}\left(e^{M/4}-1\right)\quad\text{and}\quad K_n(M)=\frac{1}{2}+\frac{2}{M}\left(e^{M/4}-1\right).
\]
\end{lemma}

\begin{proof}
The bounds on $j_1(z)$, $\jpp(z)$ and $j_2^\prime(z)$ all follow from properties of alternating series and straightforward computation to show the absolute values of their summands form decreasing sequences when $0\leq z^2\leq d+2$. We will show this computation for $j_2'(z)$; the other two are very similar.

Consider in absolute value the quotient of successive terms of $j_2'(z)$. Then by properties of the $\Gamma$ function and our bound on $z^2$, we have:
\begin{align*}
 \frac{n_{k+1}z^{2(k+1)+1}}{n_kz^{2k+1}}&=\frac{2^{-2k-2-d/2}z^{2k+3}(k+2)}{(k+1)!\Gamma(k+3+d/2)}\cdot\frac{k!\Gamma(k+2+d/2)}{(k+1)2^{-2k-d/2}z^{2k+1}}\\
 &=\frac{2^{-2}z^2(k+2)}{(k+1)^2(k+2+d/2)}\leq\frac{(k+2)(d+2)}{4(k+1)^2(k+2+d/2)}\\
 &\leq\frac{d+2}{4(k+1)^2},
\end{align*}
which is clearly nonnegative and decreasing in our index $k$.

The lower bounds on $i_1(z)$ follow immediately from the series expansion of $i_1(z)$ and the positivity of the summands. This leaves only the upper bounds on $\ipp(z)$ and $i_2'(z)$.

For $i_2'(z)$, we assume $0\leq z\leq\sqrt{M}$ and observe that
\begin{align*}
 i_2'(z)-n_0z-n_1z^3&=\sum_{k=2}^\infty\frac{2^{-2k-d/2}(k+1)}{k!\Gamma(k+2+d/2)}z^{2k+1}\\
 &=2^{d/2-3}z^3\sum_{k=2}^\infty\frac{k+1}{k!(k+1+d/2)\Gamma(k+1+d/2)}\left(\frac{z}{2}\right)^{2k-2}\\
 &\leq \frac{2^{-3-d/2}z^3}{\Gamma(3+d/2)}\sum_{k=2}^\infty\frac{k+1}{k!(k+1+d/2)}\left(\frac{z^2}{4}\right)^{k-1}\\
 &\leq\frac{1}{2}n_1b^3\sum_{k=2}^\infty\frac{1}{k!}\left(\frac{M^2}{4}\right)^{k-1}\qquad\text{since $z^2\leq M$}\\
 &=\frac{1}{2}n_1z^3\frac{4}{M}\sum_{k=2}^\infty\frac{1}{k!}\left(\frac{M^2}{4}\right)^{k}=\frac{2n_1z^3}{M}\left(e^{M/4}-1-\frac{M}{4}\right).
\end{align*}
As a result,
\[
 i_2'(z)\leq n_0z+n_1z^3+n_1\left(e^{M/4}-1-\frac{M}{4}\right)z^3,
\]
which can be simplified to obtain our desired bound on $i_2'(z)$.

The upper bound on $\ipp(z)$ can be proved similarly. When $0\leq z\leq\sqrt{M}$,
\begin{align*}
 \ipp(z)-&d_1z-d_2z^3=\sum_{k=3}^\infty \frac{2^{-d/2}(2k+1)}{(k-1)!\Gamma(k+1+d/2)}\left(\frac{z}{2}\right)^{2k-1}\\
 &\leq \frac{2^{-3-d/2}}{\Gamma(3+d/2)}z^3\sum_{k=3}^\infty\frac{2k+1}{(k-1)!(k+d/2)}\left(\frac{z}{2}\right)^{2k-4}\\
 &\leq \frac{d_2}{5}z^3\sum_{k=3}^\infty\frac{2}{(k-1)!}\left(\frac{M}{4}\right)^{k-2}\leq \frac{8d_2}{5M}z^3\sum_{k=3}^\infty\frac{1}{(k-1)!}\left(\frac{M}{4}\right)^{k-1}\\
 &=\frac{8}{5M}z^3\left(e^{M/4}-1-\frac{M}{4}\right).
\end{align*}
Thus
\[
 \ipp(z)\leq d_1z+d_1z^3+\frac{8 d_2}{5M}\left(e^{M/4}-1-\frac{M}{4}\right)z^3,
\]
which simplifies to our desired upper bound.
\end{proof}

\section{ The Natural Boundary Conditions}\label{bcsec}
In this section, our goal is to derive the form of the natural boundary conditions necessarily satisfied by all eigenfunctions. Consider the weak eigenvalue equation for eigenfunction $u$ with eigenvalue $\omega$ and some test function $\phi\in C^\infty_c(\Omega)$:
\[
 \int_\Omega (1-\sigma)\sum_{i,j}u_{x_ix_j}\phi_{x_ix_j}+\sigma\Delta u\Delta\phi+\tau Du\cdot D\phi -\omega u\phi\,dx.
\]
Because the eigenfunction $u$ is smooth, we may use integration by parts to move most of the derivatives on $\phi$ to $u$; this gives us a volume integral and two surface integrals that must vanish for all $\phi$.

We first state the natural boundary conditions for a smoothly-bounded region in arbitrary dimension:
\begin{prop} \label{generalBC} For any smoothly bounded $\Omega$, the natural boundary conditions for eigenfunctions of the free plate under tension have the form
\begin{align*}
&Mu := (1-\sigma)\frac{\partial^2 u}{\partial n^2}+\sigma\Delta u = 0 &\text{on $\partial\Omega$,}\\
&Vu := \tau\frac{\partial u}{\partial n}
-(1-\sigma)\sdiv\Big(\sproj\left[(D^2u)n\right]\Big)-\frac{\partial\Delta u}{\partial n} = 0 &\text{on $\partial\Omega$,}
\end{align*}
where $\partial/\partial n$ denotes the normal derivative and $\sdiv$ is the surface divergence. The projection $\sproj$ projects a vector $v$ at a point $x$ on $\partial\Omega$ into the tangent space of $\partial\Omega$ at $x$.
\end{prop}

When $\Omega$ is a ball, we can simplify the general boundary conditions.

\begin{prop}\label{ballBC} (Ball) In the case $\Omega$ is the ball of radius $R$, the natural boundary conditions may be written as
\begin{align}
&Mu := (1-\sigma)u_{rr}+\sigma\Delta u = 0 &\text{at $r=R$,}\label{BCb1}\\
&Vu := \tau u_r-(1-\sigma)\frac{1}{r^2}\Delta_S\left(u_r-\frac{u}{r}\right)-(\Delta u)_r = 0&\text{at $r=R$.}\label{BCb2}
\end{align}
\end{prop}

\begin{proof}[Proof of Proposition~\ref{generalBC}]
Our eigenfunctions $u$ are smooth on $\overline{\Omega}$ by regularity and satisfy the weak eigenvalue equation $a(u,\phi)-\omega(u,\phi)_{L^2(\Omega)}=0$ for all $\phi\in H^2(\Omega)$. That is,
\begin{equation}\label{weakeval}
 \int_\Omega \left((1-\sigma)\sum_{i,j=1}^du_{x_ix_j}\phi_{x_ix_j}+\sigma\Delta u\Delta\phi+\tau D\phi\cdot Du-\omega u\phi\right)\,dx=0.
\end{equation}

Much of the work is already done for us in the proof of the boundary conditions for $\sigma=0$, in \cite{inequalitypaper}.

Let $n$ denote the outward unit normal to the surface $\partial\Omega$. We can rewrite the Laplacian term by applying the Divergence theorem twice:
\begin{align*}
 \int_\Omega \Delta u \Delta \phi\,dx &=\int_{\partial\Omega} (\Delta u)\frac{\partial \phi}{\partial n}\,dS-\int_\Omega D(\Delta u)\cdot D\phi\,dx\\
&=\int_{\partial\Omega} (\Delta u)\frac{\partial \phi}{\partial n}-\phi \frac{\partial(\Delta u)}{\partial n}\,dS +\int_\Omega (\Delta^2u)\phi\,dx.
\end{align*}
Combining this with the form of the Hessian term found in the $\sigma=0$ case in \cite[Proposition 6]{inequalitypaper}, we obtain
\begin{align*}
 \int_\Omega(1-\sigma)&\sum_{i,j}u_{x_ix_j}\phi_{x_ix_j}+\sigma(\Delta u)(\Delta\phi)\,dx\\
&=\int_\Omega (\Delta^2u)\phi\,dx+\int_{\partial\Omega}\frac{\partial\phi}{\partial n}\left((1-\sigma)\frac{\partial^2u}{\partial n^2}+\sigma\Delta u\right)\,dS\\
&\qquad-\int_{\partial\Omega}\phi\left(\frac{\partial(\Delta u)}{\partial n}+(1-\sigma)\sdiv\left(\sproj\left[(D^2u)n\right]\right)\right)\,dS.
\end{align*}
Thus for $u$ an eigenfunction associated with eigenvalue $\omega$, we see~\eqref{weakeval} can be written as
\begin{align*}
0&=\int_\Omega\phi\Big(\Delta^2 u-\tau\Delta u-\omega u\Big)\,dx
+\int_{\partial\Omega} \frac{\partial\phi}{\partial n}\left((1-\sigma) \frac{\partial^2u}{\partial n^2}+\sigma\Delta u\right)\,dS\\
&\qquad+\int_{\partial\Omega}\phi\left(\tau\frac{\partial u}{\partial n}-\frac{\partial\Delta u}{\partial n}-(1-\sigma)\sdiv\Big(\sproj\left[(D^2u)n\right]\Big) \right)\,dS.
\end{align*}

As in the membrane case, this identity must hold for all $\phi\in H^2(\Omega)$. If we take any compactly supported $\phi$, then the volume integral must vanish; because $\phi$ is arbitrary, we must therefore have $\Delta^2 u-\tau\Delta u-\omega u=0$ everywhere. Similarly, the terms multiplied by $\phi$ and $\partial\phi/\partial n$ must vanish on the boundary. Collecting these results, we obtain the eigenvalue equation \eqref{maineq} and natural boundary conditions of Proposition~\ref{generalBC}.
\end{proof}

Once we have the general form of the boundary conditions, we can find their expression in spherical coordinates when $\Omega$ is a ball, from which Proposition~\ref{ballBC} follows directly. The necessary computations were performed in the proof of \cite[Proposition 7]{inequalitypaper}; we do not repeat them here.

\section{ The eigenmodes of the ball}\label{fundtonesec}
The ball is a rare case in which we can find exact solutions for the vibrating plate. As in the $\sigma=0$ case treated in \cite{besselpaper}, we factor the eigenvalue equation:
\[
 (\Delta+a^2)(\Delta-b^2)u=0,\text{~~where $b^2=a^2+t$ and $a^2b^2=\omega$.}
\]
After writing the factors $(\Delta+a^2)$ and $(\Delta-b^2)$ in spherical coordinates, we are able to write the eigenfunctions $u$ as
\[
 u(r,\hat{\theta})=\Big(j_l(a r)+\gamma_l i_l(b r)\Big)Y_l(\hat{\theta})
\]
where $Y_l$ is an $l$th-order spherical harmonic and $\gamma$ is a coefficient determined by the boundary conditions. There are two boundary conditions, so we could use either to express $\gamma$; we will find it more convenient to write $\gamma_l=-Mj_l(a)/Mi_l(b)$.

We are able to exclude Bessel functions of the second kind $n_l$ and modified Bessel functions $k_l$ of the second kind because they are singular at the origin but of different order for a fixed $l$, and so no linear combination of $n_l$ and $k_l$ will be continuous at the origin. See, eg, \cite{besselpaper,cthesis}.

\begin{prop}\label{eigfc} (Eigenfunctions in spherical coordinates) Let $\tau>0$ and $\omega$ be any positive eigenvalue of the free ball $\BB(R)$; that is, $\omega$ is an eigenvalue of $\Delta\Delta u-\tau\Delta u =\omega u$ under boundary conditions \eqref{BCb1} and \eqref{BCb2}. Then the corresponding eigenfunctions can be written in the form $R_l(r)\Yl(\thetahat)$, where $\Yl$ is a spherical harmonic of some integer order $l$ and $R_l$ is a linear combination of ultraspherical Bessel and modified Bessel functions,
\[
R_l(r) = j_l(ar/R)+\gamma i_l(br/R).
\]
Here the positive numbers $a$ and $b$ depend on $\tau$ and $\omega$ by $b^2-a^2=R^2\tau$ and $a^2b^2=R^4\omega$, and $\gamma$ is a real constant given by
\[
 \gamma=\frac{-Mj_l(a)}{Mi_l(b)}.
\]
\end{prop}

The proof of this proposition is almost identical to the proof of \cite[??]{besselpaper} and so is not repeated here. It proceeds roughly as follows: We argue that $\Delta^2-\tau\Delta$ and $\Delta_S$ are simultaneously diagonalizable to justify factoring the eigenvalue equation and writing solutions as a product of radial and angular parts, with the radial part being a linear combination of Bessel and Modified Bessel functions of the first and second kind. The regularity of eigenfunctions is used to conclude the coefficients of the singular second-kind Bessel functions must be zero, and finally, we use the boundary condition $Mu=0$ when $r=1$ to find $\gamma$.

Based on the $\sigma=0$ case and other supporting evidence, we make the following
\begin{conj} For $\tau>0$ and $\sigma\in[0,1)$, the fundamental modes of the ball $\BB(R)$ can be written as linear combinations of
\[
u_1(r,\thetahat)=\Big(j_1(ar/R)+\gamma i_1(br/R)\Big)Y_1(\thetahat),
\]
with $a$, $b$, $\gamma$ real constants, with $a$ and $b$ positive and depending on $\tau$, $\sigma$ and $\omega_1$ as follows: $b^2-a^2=R^2\tau$ and $a^2b^2=R^4\omega_1(\tau,\sigma)$, and $\gamma$ given by
\[
 \gamma=\frac{-Mj_1(a)}{Mi_1(b)}.
\]
\end{conj}
In this section we will prove a weaker result and an ancillary lemma, and conclude with a more thorough discussion of the evidence supporting the conjecture.This  conjecture was proved for $\sigma=0$ in \cite[Theorem 3]{besselpaper}, treating cases of small and larger index $l$ separately. First, we showed that for $l\geq 1$, the quotient $Q[R\Yl]$ was an increasing function of $l$ for any radial function $R$; this approach fails for $\sigma>0$ because of the inclusion of the term $(\Delta_r R)^2$, which cannot be rewritten to be monotone in $l$ for all admissable $R$. This may be due to the lack of coercivity of the quotient $\int_\Omega(\Delta u)^2\,dx/\int_\Omega u^2\,dx$. 

The second part of the proof, showing that the first eigenvalue for modes with index $l=1$ is lower than the first nonzero eigenvalue for radially symmetric modes (index $l=0$), can be adapted to nonzero values of $\sigma$. As a result, we have the following

\begin{prop}\label{fundmode} For $\tau>0$ and $\sigma\in[0,1)$, the fundamental modes of the ball $\BB(R)$ can be written as linear combinations of
\[
u_1(r,\thetahat)=\Big(j_l(ar/R)+\gamma i_l(br/R)\Big)Y_l(\thetahat),
\]
where the index $l\geq 1$ and with $a$, $b$, $\gamma$ real constants, with $a$ and $b$ positive and depending on $\tau$, $\sigma$ and $\omega_1$ as follows: $b^2-a^2=R^2\tau$ and $a^2b^2=R^4\omega_1(\tau,\sigma)$, and $\gamma$ given by
\[
 \gamma=\frac{-Mj_l(a)}{Mi_l(b)}.
\]
\end{prop}

\begin{proof}
Let $\pll$ denote the first positive zero of $j_1'(a)$. Recall from Proposition~\ref{eigfc} that $\omega$ is an eigenvalue if and only if we have some integer $l\geq 0$ and can write $\omega=a^2b^2$ for positive constants $a$ and $b$ such that $b^2-a^2=\tau$ and $W_l(a)=0$. The parameter $\tau$ is positive, so $\omega=a^2(a^2+\tau)$ increases with $a$. Therefore, to show that the lowest nonzero eigenvalue corresponds to $l=1$ and not $l=0$, we show that the first nonzero root of $W_1(a)$ is less than the first nonzero root of $W_0(a)$.

First we consider $l=1$. We will show that $W_1(a)$ changes sign on the interval $(0,\pll)$. Note first that by Lemma~\ref{Wlprops}, the function $W_1(a)$ is negative as $a\to0^+$.

We next show that $W_1(\pll)>0$; then by continuity we will have shown $W_1(a)$ contains a root in the interval $(0,\pll)$. Immediately from its definition, we can write:
\begin{align*}
W_1(a)&=\Big(((1-\sigma)(d-1)-a^2)j_1(a)-(1-\sigma)(d-1)aj_1^\prime(a)\Big)\\
&\qquad\qquad\cdot\Big(-a^2bi_1^\prime(b)+(1-\sigma)(d-1)(bi_1^\prime(b)-i_1(b))\Big)\\
&\quad-\Big(((1-\sigma)(d-1)+b^2)i_1(b)-(1-\sigma)(d-1)bi_1^\prime(b)\Big)\\&\qquad\qquad\cdot\Big(ab^2j_1^\prime(a)+(1-\sigma)(d-1)(aj_1^\prime(a)-j_1(a))\Big).
\end{align*}
We have $j_1'(\pll)=0$ by definition of $\pll$, so $W_1(\pll)$ simplifies considerably. Factoring out $j_1(\pll)$, we find that
\begin{align*}
 \frac{1}{j_1(\pll)}W_1(\pll)&=(1-\sigma)(d-1)Mi_1(b)+\Big(\pllsq-(1-\sigma)(d-1)\Big)^2bi_1'(b)\\
 &\qquad+\Big(\pllsq-(1-\sigma)(d-1)\Big)(1-\sigma)(d-1)i_1(b).
\end{align*}
Note that $0<(1-\sigma)(d-1)<d$, so
\[
 \pllsq-(1-\sigma)(d-1)>\pllsq-d>0
\]
since $\pllsq>d$ by Proposition~\ref{propLS}. Thus every summand in $W_1(\pll)/j_1(\pll)$ is positive. Since $j_1(\pll)>0$, this means that $W_1(\pll)>0$, as desired.

Now we consider $l=0$. Here $K_l = l(l+d-2)$ vanishes, so we look for $a$ solving:
\begin{align*}
0&=W_0(a)\\ &=\Big((1-\sigma)a^2j_0^{\prime\prime}(a)-\sigma a^2j_0(a)\Big)\Big(-a^2bi_0^{\prime}(b)\Big)\\
&\qquad-\Big((1-\sigma)b^2i_0^{\prime\prime}(b)+\sigma b^2i_0(b)\Big)\Big(ab^2j_0^\prime(a)\Big)\\
&=\Big((1-\sigma)a^2j_1^\prime(a)+\sigma a^2j_0(a)\Big)a^2bi_1(b)\\
&\qquad+\Big((1-\sigma)b^2i_1^\prime(b)+\sigma b^2i_0(b)\Big)ab^2j_1(a),
\end{align*}
by \eqref{j2} and \eqref{i2} in Lemma~\ref{Besselprops}. The functions $i_0(b)$ and $i_1^\prime(b)$ are positive for $b>0$ by their power series. Similarly, $j_1(a)$ and $j_1'(a)$ are positive on $(0,\pll)$ by Lemma~\ref{Besselsigns}, and so for $\sigma\geq0$, we have shown $W_0(a) > 0$ on $(0,\pll)$.

For $\sigma<0$, we use identities \eqref{j1} and \eqref{i1} in Lemma~\ref{Besselprops} to write
\begin{align*}
 W_0(a)&=\Big((1-\sigma)a^2j_1^{\prime}(a)+\sigma a^2j_1^{\prime}(a)+\sigma a(d-1)j_1(a)\Big)a^2bi_1(b)\\
&\qquad+\Big((1-\sigma)b^2i_1^\prime(b)+\sigma b^2i_1(b)-(d-1)\sigma bi_1(b)\Big)ab^2j_1(a)\\
&=a^4j_1'(a)bi_1(b)+aj_1(a)b^4i_1'(b)+\tau\sigma(d-1)aj_1(a)bi_1(b).\qedhere
\end{align*}
\end{proof}

\begin{lemma}\label{Wlprops}
For any dimension $d\geq2$, the function $W_l(a)$ is convex as a function of $\sigma$ for any $l\geq 2$ and linear as a function of $\sigma$ for $l=0,1$.

For any dimension $d\geq 2$, index $l\geq 1$, $\sigma\in[-1/(d-1),1]$, and positive $\tau$, the function $W_l(a)$ is negative as $a\to0^+$.
\end{lemma}
\begin{proof} First, we shall establish convexity or linearity of $W_l(a)$ as a function of $\sigma$. Treating $a, \tau$, and $\sigma$ all as independent variables, we differentiate twice with respect to $\sigma$, obtaining:
\begin{align*}
\frac{\partial^2}{\partial\sigma^2}W_l(a)&=2K_l\Big(a^2j_l^{\prime\prime}(a)+a^2j_l(a)\Big)\Big(bi_l(b)-i_l(b)\Big)\\
&\qquad+2K_l\Big(-b^2i_l^{\prime\prime}(b)+b^2i_l(b)\Big)\Big(aj_l(a)-j_l(a)\Big)\\
&=2K_l\Big(K_lj_l(a)-(d-1)aj_l^\prime(a)\Big)\Big(bi_l(b)-i_l(b)\Big)\\
&\qquad+2K_l\Big(-K_li_l(b)+(d-1)bi_l^\prime(b)\Big)\Big(aj_l(a)-j_l(a)\Big)\\
&=K_l(K_l-(d-1))\Big(j_l(a)bi_{l+1}(b)+i_l(b)a j_{l+1}(a)\Big).
\end{align*}
Since $K_l=l(l+d-2)$, we have that $K_0=0$ and $K_l=(d-1)$, so this quantity vanishes for those two indices. Otherwise, $K_l>(d-1)$ and so by our knowledge of signs of Bessel functions from Lemma~\ref{Besselsigns}, the above is positive for indices $l\geq2$. Hence $W_l(a)$ is linear in $\sigma$ for $l=0,1$ and convex in $\sigma$ for $l\geq2$, as desired.

Next, we establish the negativity of $W_l(a)$ as $a\to0^+$. By the series expansions of Bessel functions, we have that as $a\to0^+$,
\begin{align*}
 j_l(a)&=c_{l,0}a^l-c_{l,1}a^{l+2}+\BigO(a^{l+4})\\
 j_l(a)-aj_l'(a)&=-(l-1)c_{l,0}a^l+(l+1)c_{l,1}a^{l+2}+\BigO(a^{l+4})\\
 a^2j_l^{\prime\prime}(a)&=l(l-1)c_{l,0}a^l-(l+2)(l+1)c_{l,1}a^{l+2}+\BigO(a^{l+4}),
\end{align*}
and since $b=\sqrt{a^2+\tau}$, we have that as $a\to0^+$,
\begin{align*}
 i_l(b)&=i_l(\sqrt{\tau})+\BigO(a^2),\\
 bi_l^\prime(b)&=\sqrt{\tau}i_l^\prime(\sqrt{\tau})+\BigO(a^2),\\
  b^2 i_l^{\prime\prime}(b)&=\tau i_l^{\prime\prime}(\sqrt{\tau})+\BigO(a^2).
\end{align*}
Thus for small $a$ values,
\begin{align*}
 Mj_l(a)&=(1-\sigma)l(l-1)c_{l,0}a^l+\BigO(a^{l+2}),\\
 Vj_l(a)&=(l\tau +(1-\sigma)K_l(l-1))c_{l,0}a^l+\BigO(a^{l+2}),\\
 Mi_l(b)&=Mi_l(\sqrt{\tau})+\BigO(a^{2}),\\
 Vi_l(b)&=(1-\sigma)K_l\Big(\sqrt{\tau}i_l'(\sqrt{\tau})-i_l(\sqrt{\tau})\Big)+\BigO(a^2),
\end{align*}
and so as $a\to0^+$
\begin{align*}
 W_l(a)&=
 (1-\sigma)l(l-1)(1-\sigma)K_l\Big(\sqrt{\tau}i_1'(\sqrt{\tau})-i_1(\sqrt{\tau})\Big)c_{l,0}a^l,\\
 &\qquad-Mi_l(\sqrt{\tau})(l\tau +(1-\sigma)K_l(l-1))c_{l,0}a^l+\BigO(a^{l+2}).
\end{align*}
Since $W_l(a)$ is either linear or convex in $\sigma$, it will be maximized at one of the two extreme values of $\sigma$. Hence it suffices to show $W_l(a)<0$ as $a\to0^+$ for both $\sigma=1$ and $\sigma=-1/(d-1)$.

If $\sigma=1$, then $Mi_l(z)=-z^2 i_l(z)$ and so our bound on $W_l$ simplifies to:
\[
 W_l(a)=-\tau^2 i_l(\sqrt{\tau})l c_{l,0}a^l+\BigO(a^{l+2}).
\]
Since $a^l$ and $c_{l,0}$ are both nonnegative, by positivity of the Bessel $i_l$ functions, $W_l(a)<0$ for sufficiently small $a$.

If $\sigma=1$, then $1-\sigma=d/(d-1)$ and we can write \\$(d-1)Mi_l(z)=dz^2 i_l^{\prime\prime}(z)-z^2i_l(z)$, and our bound on $W_l$ simplifies to:
\begin{align*}
 (d-&1)^2W_l(a)=
 d^2l(l-1)K_l\Big(\sqrt{\tau}i_1'(\sqrt{\tau})-i_1(\sqrt{\tau})\Big)c_{l,0}a^l\\
 &\qquad-\Big(d\tau i_l^{\prime\prime}(\sqrt{\tau})-\tau i_l(\sqrt{\tau})\Big)(l\tau(d-1)+dK_l(l-1))c_{l,0}a^l+\BigO(a^{l+2}).
\end{align*} 
Note $c_{l,0}$ is positive and so does not affect the sign. The coefficient of the $c_{l,0}a^l$ terms above can be rewritten using Bessel identities, yielding: 
\begin{align*}
&\frac{d^2l(l-1)K_l}{l+d-1}\Big(((l-1)(l+d-1)+\tau)i_l(\sqrt{\tau})-\tau i_{l+1}^\prime(\sqrt{\tau})\Big)\\
 &\quad-\frac{l\tau(d-1)+d(l-1)K_l}{l+d-1}\Big(dl(l-1)(l+d-1)+(d-1)(l-1)\tau\Big)i_l(\sqrt{\tau})\\
 &\quad-\frac{l\tau(d-1)+d(l-1)K_l}{l+d-1}d(d-1)\tau i_{l+1}^\prime(\sqrt{\tau})\\
 &=-\frac{l(l-1)(d-1)^2\tau^2+(d-1)d(l-1)^2(l+d-1)\tau}{l+d-1} i_l(\sqrt{\tau})\\
 &\quad-\frac{dl(d-1)^2\tau^2+d^2(l-1)K_l\tau}{l+d-1}i_{l+1}^\prime(\sqrt{\tau}),
\end{align*}
which is negative for all $l,d$, and $\tau$ under consideration. Thus for sufficiently small values of $a$, we have $W_l(a)<0$ when $\sigma=-1/(d-1)$, as desired.
\end{proof}

\subsection*{Evidence for the conjecture.}
We now discuss the body of evidence for our conjecture.

When $\sigma=0$, it has already been proved that the fundamental mode of the ball corresponds to index $l=1$ and has simple angular dependence \cite[Theorem 3]{besselpaper} and so the conjecture is true in this case.

\subsubsection*{Limiting cases} The limiting case of zero tension. By Lemma~\ref{wbounds}, we have for positive tension that $\tau\mu_1\leq\omega_1(\tau,\sigma)\leq\tau(d+2)$; taking the limit as $\tau\to0^+$ gives us that $\omega_1(\tau,\sigma)\to0$ in this limit. 
 
When $\tau=0$, the differential eigenvalue equation becomes $\Delta^2u =\omega u$ and constant and linear functions are all eigenfunctions, and so the eigenvalue $\omega=0$ has $d+1$ multiplicity. By taking the spanning set to be $u=\text{const}$ and $u=x_k$, $k=1,\dots,d$, we see the nonconstant eigenfunctions have simple angular dependence, which would be the limiting case of eigenfunctions of the form $R(r)Y_1(\thetahat)$.
  
The limiting case of infinite tension (zero rigidity). We saw in Corollary~\ref{limwt} that $\omega_1/\tau$ approaches the fundamental tone of the free membrane as $\tau\to\infty$; the fundamental modes of the free membrane are of the form $u(r,\thetahat)=j_1(\pll r/R)Y_1(\thetahat)$. 
 
From Proposition~\ref{fundmode}, the radial part of the fundamental mode can be written as $j_l(ar)+\gamma i_l(br)$ with $\gamma=-Mj_l/Mi_l$; then in the limit as $\tau\to\infty$, $\gamma\to0$ and so we obtain eigenfunctions of the form $j_l(ar)Y_l(\thetahat)$. Taking $l=1$ would give us agreement with the free membrane.
 
 \subsubsection*{Numerical evidence}
One consequence of Lemma~\ref{Wlprops} is that it allows us to reduce the number of parameters we need to consider when numerically verifying that $l=1$ gives us the fundamental mode. The convexity in $\sigma$ means that for any fixed $a$, the value of $W_l(a)$ is maximal at either $\sigma=-1/(d-1)$ or $\sigma=1$. Because $W_l(0)=0$ and is decreasing and negative for small $a$, we may conclude that for a fixed $l$, the smallest first root $a_l^*$ of $W_l$ occurs at either $\sigma=-1/(d-1)$ or $\sigma=1$. On the other hand, because $W_1(a)$ is linear in $\sigma$, the \emph{largest} value of the first root $a_1^*$ occurs at either $\sigma=-1/(d-1)$ or $\sigma=1$.

Thus if we can show that the for any fixed dimension $d$, any $\tau>0$, and any index $l\geq2$, the first roots of $W_1(a,\sigma=1)$ and $W_1(a,\sigma=-1/(d-1))$ are smaller than those of $W_l(a,\sigma=1)$ and $W_l(a,\sigma=-1/(d-1))$, then we will have proved that the lowest positive eigenvalue does correspond to the index $l=1$.

Because of the sheer complexity of the $W_l(a)$ functions it does not seem to be possible to prove this directly. However, it is easy to verify this numerically for any choice of $d$, $l$, and $\tau$. 
Numerical investigations suggest that for any dimension $d\geq2$, any tension $\tau>0$ and any $l\geq2$, the function $W_l(a,\sigma=-1/(d-1))<0$ for $a\in(0,\pll)$ and that it suffices to consider the $\sigma=1$ case for $l\geq 2$. This is demonstrated in Figure~\ref{negWL}. The graphs were generated using LogLogPlot in Mathematica; we provide a source file on the ArXiv.
\begin{figure}[h!]\label{negWL}
 \begin{center}
  \includegraphics[width=2.45in]{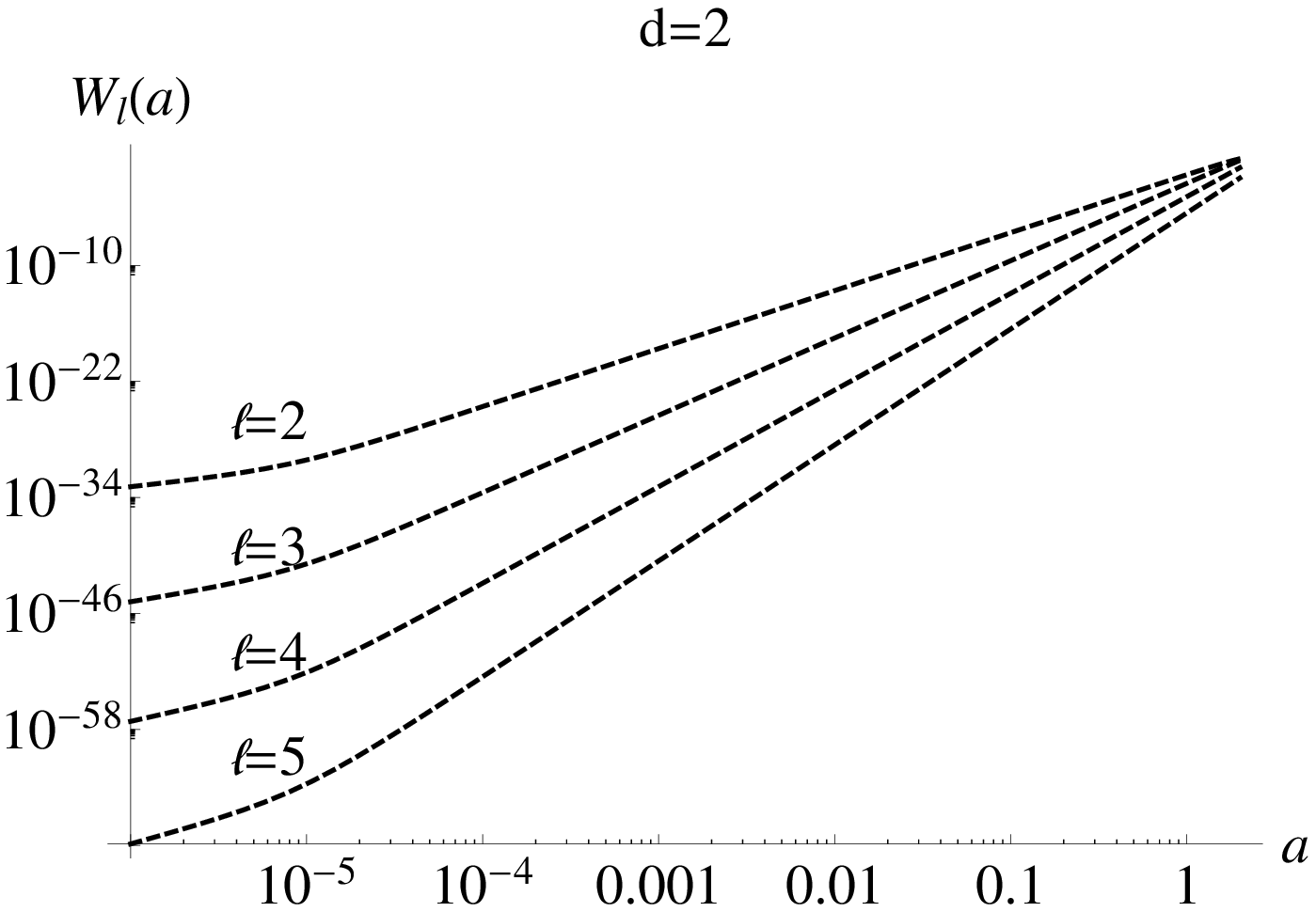}
   \includegraphics[width=2.35in]{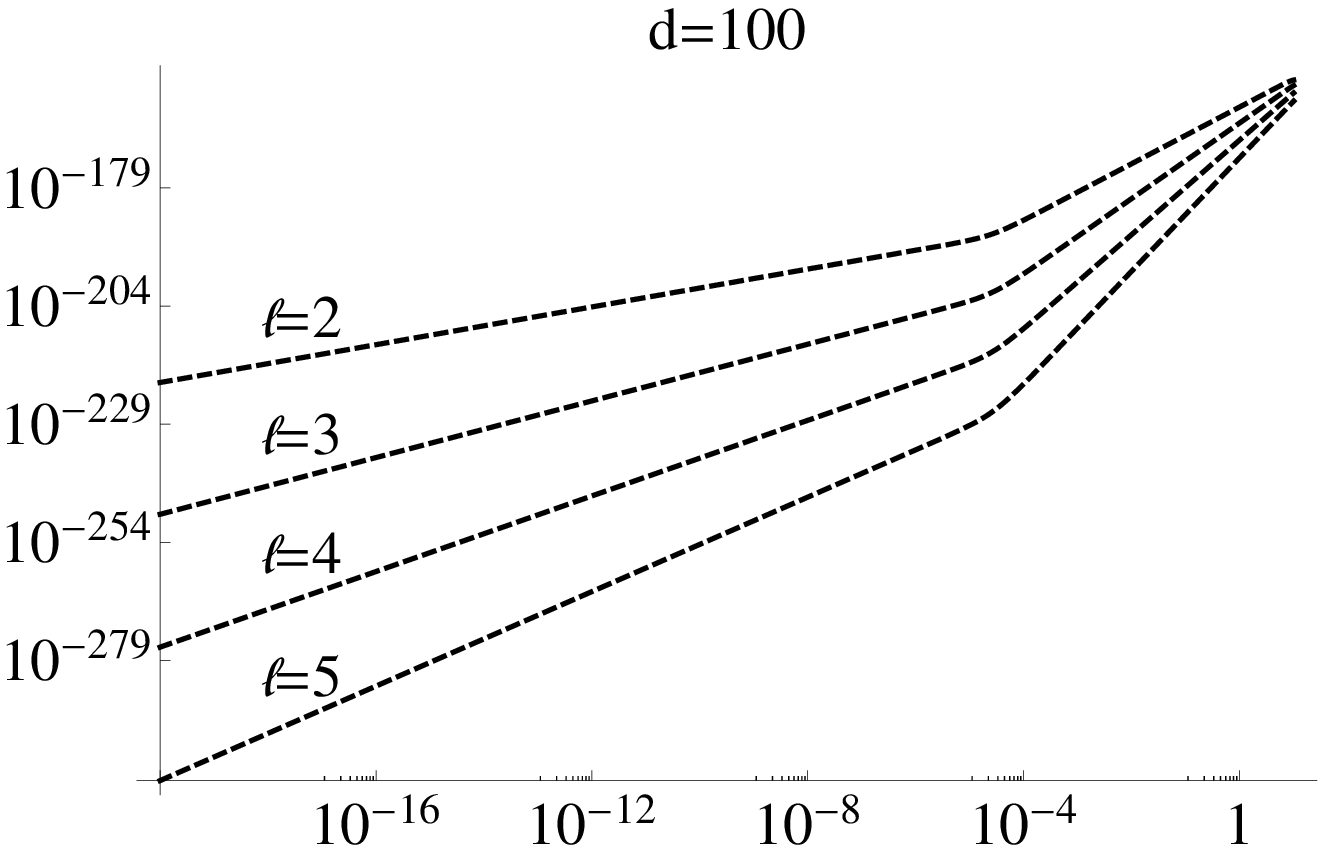}
   \caption{Log-log plots of $-W_l(a)$ for indices $l=2,3,4,5$ and $\tau=1$ for dimensions $d=2$ (top) and $d=100$ (bottom) for $a\in[10^{-20},\sqrt{d+2}]$ and $\tau=10^{-10}$.  Solid curves are $l=2$; dashed are $l=3$, dotted are $l=4$, and dash-dotted are $l=5$. From these images we see that $-W_l(a)$ remains positive even for very small $\tau$ and $a$ values.}
 \end{center}
\end{figure}

Figure~\ref{wrootspic} shows the roots $a_l^*$ as functions of $\tau$ for various dimensions $d$ and indices $l$. These images were produced in Mathematica using ContourPlot and LogLogPlot; we provide a source file on the ArXiv. The thick lines correspond to the roots $a$ of $W_1(a,\sigma=1)$ and $W_1(a,\sigma=-1/(d-1))$ for various dimensions, while the thinner lines are those of higher indices. Log-log plots are used for small values of $\tau$ so that the separation between the curves is more apparent.

\begin{figure}\label{wrootspic}
 \includegraphics[scale=.5]{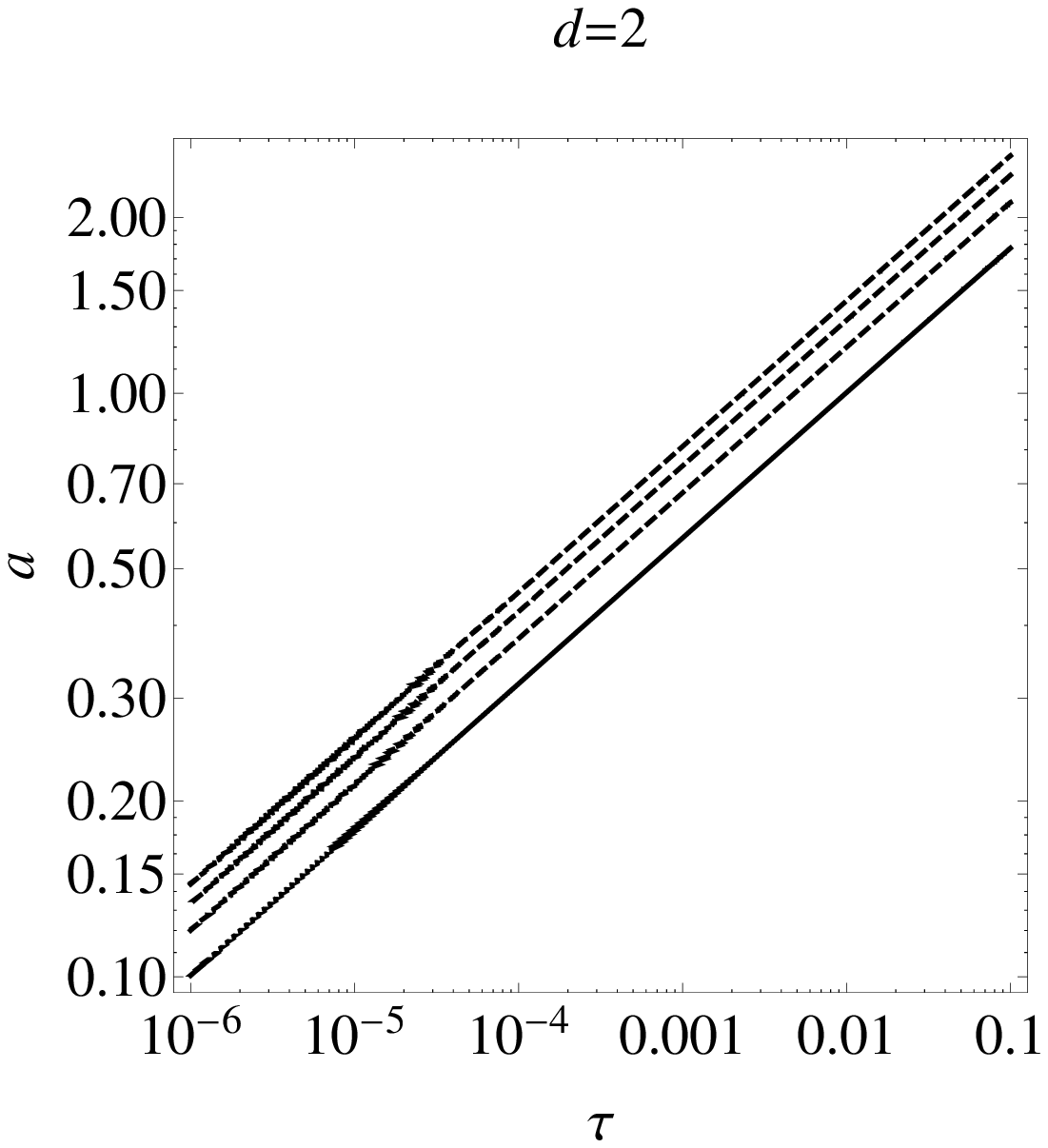}\quad\includegraphics[scale=.6]{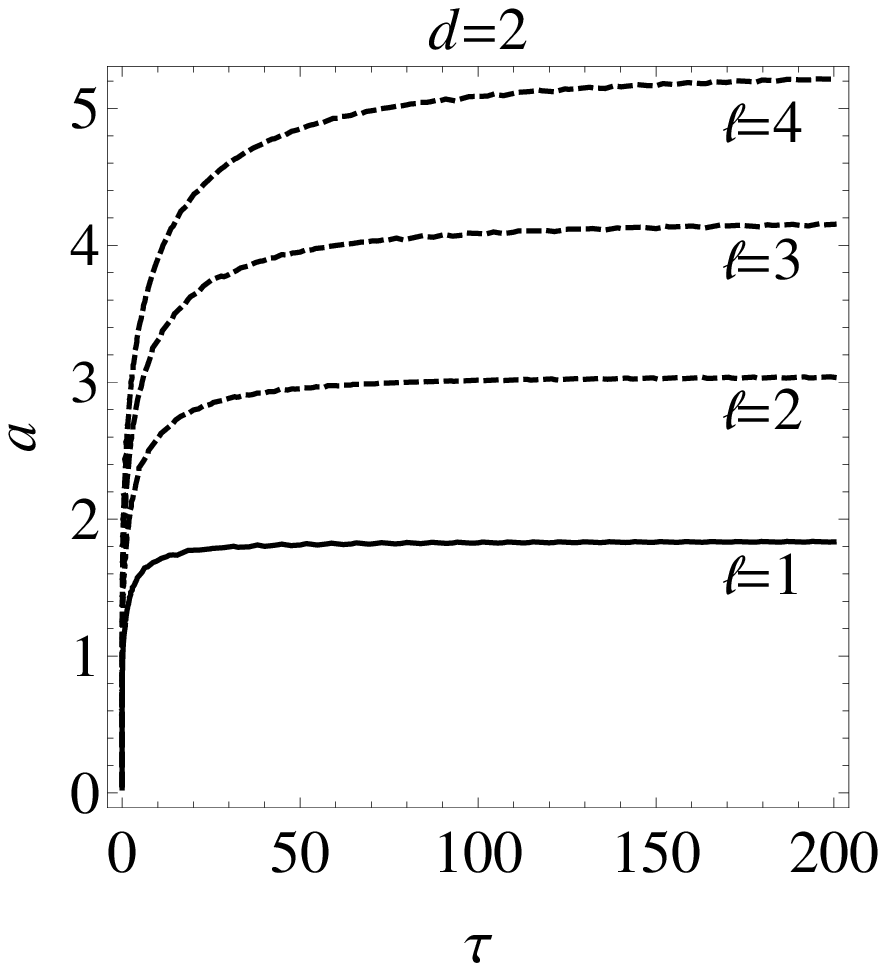}
\\
 \includegraphics[scale=.5]{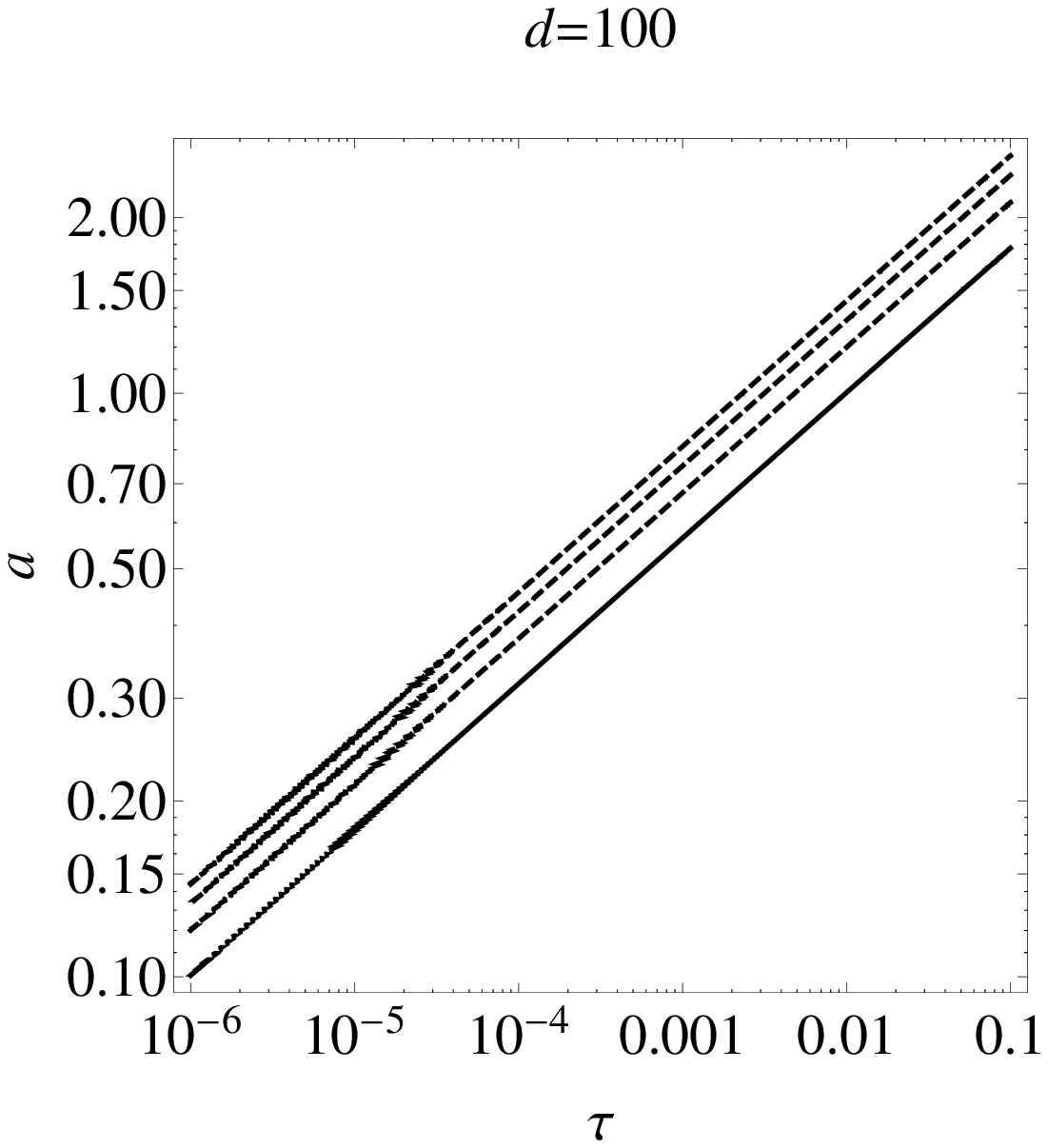}\quad\includegraphics[scale=.6]{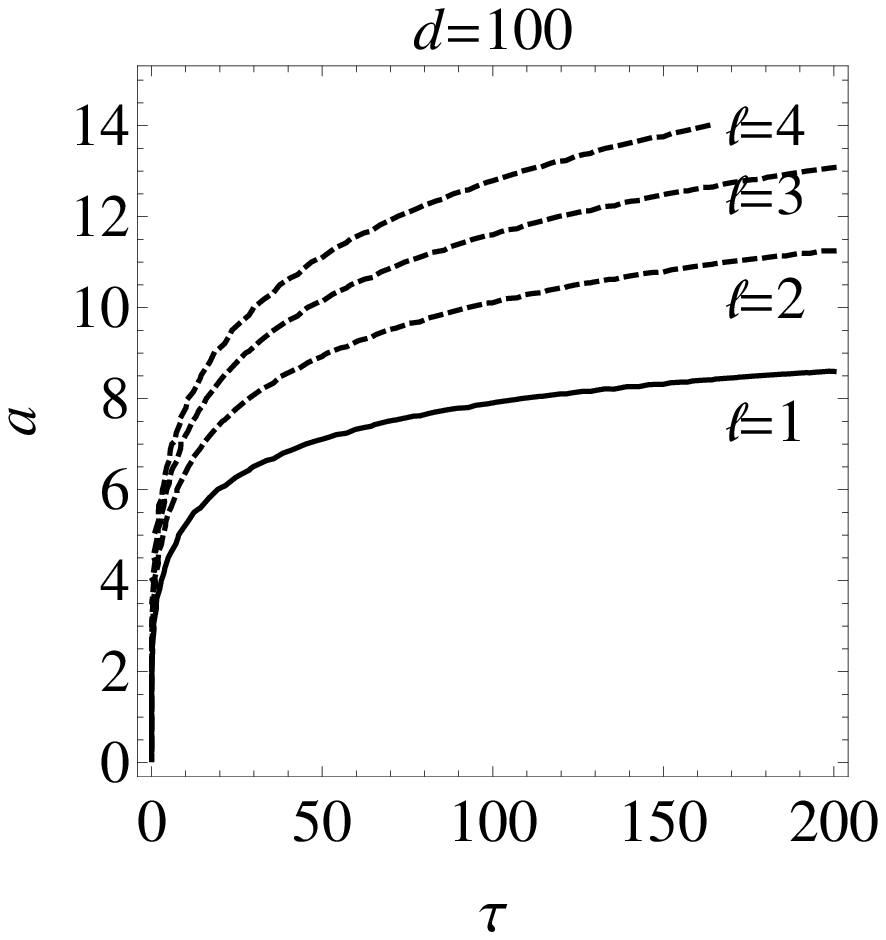}
  \caption{Graph of the first nontrivial root $a$ of $W_l(a)$ as functions of $\tau$ for extremal $\sigma$ values for dimensions $d=2$ and $d=100$. In each image, the solid curve corresponds to $l=1$ with $\sigma=0$; the dashed curves are $l=2,3,4$ with $\sigma=1$.}
\end{figure}

We end with a final useful inequality relating $\tau$ and $a$:
\begin{lemma}\label{atbounds}
For all dimensions $d\geq 2$ and all values $\tau>0$ and $\sigma\in(-1/(d-1),1)$, if $a$ is as in \ref{eigfc}, we have 
\[
 \frac{a^4}{d+2-a^2}\leq\tau.
\]
\end{lemma}
\begin{proof} This follows directly from writing $\omega_1=a^2(a^2+\tau)$ in the lower bound from Lemma~\ref{wbounds} and solving the inequality for $\tau$.
\end{proof}

\section{ Trial functions}\label{trialfcnsec}
Because we are using a trial function argument to prove Theorem~\ref{mainthm}, we will need to define these functions and establish some properties that will be useful for proving our main theorem.

Proceeding from our assumption that $l=1$ corresponds to the fundamental mode for the unit ball, we write $R(r)=j_1(ar)+\gamma i_1(br)$ for the radial portion of any fundamental mode eigenfunction.

\begin{lemma}\label{trialfcn} (Trial functions) Let the radial function $\rho$ be given by the function $R$, extended linearly. That is,
\[
 \rho(r)=\begin{cases} R(r) &\text{when $0\leq r\leq1$,}\\
       R(1)+(r-1)R'(1)&\text{when $r\geq1$.}\\
      \end{cases}
\]
After translating $\Omega$ suitably, the functions $u_k = x_k\rho(r)/r$, for  $k=1,\dots,d$, are valid trial functions for the fundamental tone.
\end{lemma}
\begin{proof} By construction, we have $x_k\rho(r)$ is continuous; then the functions $u_k$ are in $H^2(\Omega)$ provided there is no singularity introduced at the origin when we divide by $r$. By series expansions of $j_1(ar)$ and $i_1(br)$, we have $\rho(r)~c_{1,0}(a+\gamma b)r$ as $r\to0^+$, and so $R(r)/r\to0$ as $r\to0^+$. 

The $u_k$ must also be orthogonal to a constant function in order to be admissible trial functions. To achieve this we use ``center of mass'' coordinates so that $\int_\Omega u_k\,dx=0$ for $k=1,\dots,d$. This argument relies on the Brouwer Fixed Point Theorem and is identical to that in the proof of \cite[Lemma 13]{inequalitypaper}. 
\end{proof}

In the remainder of this section, we will establish several facts about the behavior of $\rho(r)$ along with bounds on the constant $\gamma$. For convenience, we write $\delr:=\rpp-\frac{d-1}{r^2}(\rho-r\rp)$. Note that $\Delta(\rho Y_1)=\delr Y_1$, so $\delr$ can be thought of as the radial part of the Laplacian.

\begin{lemma} \label{gammapos} Fix dimension $d\geq$ and constants $a\in(0,\pll)$ and $\tau>0$. Then $\gamma$ is increasing function of $\sigma$.

Furthermore, for $\sigma\in[-1/(d-1),1]$, the constant $\gamma$ satisfies the bounds 
\[
 0\leq\frac{a^2j_2'(a)}{b^2i_2'(b)}\leq\gamma\leq \frac{a^3}{b^3}<1,
\]
and when $\sigma\geq0$ we also have
\[
 \frac{-a^2\jpp(a)}{b^2\ipp(b)}\leq\gamma.
\]
\end{lemma}

\begin{proof} We compute the derivative directly and simplify:
\[
\frac{\partial\gamma}{\partial \sigma}=\frac{a^2j_1(a)b^2\ipp(b)+a^2\jpp(a)b^2i_1(b)}{((1-\sigma)b^2\ipp(b)+\sigma b^2i_1(b))^2}.
\]
The denominator is always nonnegative, so the sign of $\partial\gamma/\partial\sigma$ is determined by the numerator. Using identities from Lemma~\ref{Besselprops} first to rewrite $\jpp$ and $\ipp$, and then to rewrite $j_2$ and $i_2$ in terms of $j_1, j_3, i_1$, and $i_3$, we have
\begin{align*}
 a^2j_1(a)&b^2i_1''(b)+a^2j_1''(a)b^2i_1(b)\\
 &=(d-1)\Big(aj_2(a)b^2i_1(b)-a^2j_1(a)bi_2(b)\Big)\\
 &=\frac{d-1}{d+2}\Big(a^2(j_1(a)+j_3(a))b^2i_1(b)-a^2j_1(a)b^2(i_1(b)-i_3(b))\Big)\\
 &=\frac{d-1}{d+2}a^2b^2\Big(j_3(a)i_1(b)+j_1(a)i_3(b)\Big),
\end{align*}
which is nonnegative for $a\in[0,p_{1,1}]$ by our knowledge of signs of Bessel functions from Lemma~\ref{Besselsigns}. Thus $\gamma$ is increasing in $\sigma$ for fixed $a$ and $\tau$, as desired.

Now we note that for $\sigma\in(-1/(d-1),1)$, we have
\[
 \gamma(\sigma)\leq \gamma(1)=\frac{a^2j_1(a)}{b^2i_1(b)}\leq\frac{a^3}{b^3} 
\]
with the last inequality by Lemma~\ref{bblem}.

Similarly, we obtain
\begin{align*}
 \gamma(\sigma)&\geq \gamma\left(-\frac{1}{d-1}\right)=\frac{-da^2\jpp(a)-a^2j_1(a)}{db^2\ipp(b)-b^2i_1(b)}\\
 &=\frac{(d-1)a^2j_1(a)-d(d-1)aj_2(a)}{(d-1)b^2i_1(b)-d(d-1)bi_2(b)}=\frac{a^2j_2'(a)}{b^2 i_2'(b)},
\end{align*}
again using Bessel identities from Lemma~\ref{Besselprops} to simplify. This quotient is positive, so we have proved $\gamma\geq0$ as desired.

Finally, when $\sigma\geq0$, we observe
\[
 \gamma(\sigma)\geq\gamma(0)=\frac{-a^2\jpp(a)}{b^2\ipp(b)}.\qedhere
\]
\end{proof}

Now that we have bounds on $\gamma$, we may use these to establish some useful properties of our trial function $\rho$ and its derivatives.

\begin{lemma} \label{rhoproperties} For all $\tau>0$ and $\sigma\in(-1/(d-1),1)$, there exists a unique $\rstar\in(0,1]$ such that $\rpp<0$ on $(0,\rstar)$ and $\rpp\geq0$ on $[\rstar,1]$.

Additionally, we have that $\rho-r\rp\geq0$, $\tau\rp-(\delr)_r\geq0$ and $\delr\rho\leq0$ for all $r\in[0,1]$.
\end{lemma}

\begin{proof}
First note that the function $\rpp$ is convex in $r$ on $[0,1]$, since
\[
 \frac{d^2}{dr^2}\rpp(r)=a^4j_1^{(4)}(ar)+\gamma b^4i_1^{(4)}(br),
\]
which is positive by Lemmas~\ref{gammapos} and~\ref{Besselsigns}.

By differentiating the series expansions for $j_1$ and $i_1$ from Lemma~\ref{bblem}, we see that as $r\rightarrow0$, we have
\begin{align*}
 \rpp(r)&\sim a^2(-c_{1,1}ar+6c_{1,2}a^3r^3)+\gamma b^2(c_{1,1}br+6c_{1,2}b^3r^3)\\
 &\sim(b^3\gamma-a^3)c_{1,1}r,
\end{align*}
which is negative by Lemma~\ref{gammapos}. So by convexity, $\rpp$ is either negative on all of $(0,1]$ or has a single, simple root in that interval. Let $\rstar$ denote the root if it exists; otherwise set $\rstar=1$. Then $\rpp\leq0$ on $[0,\rstar)$ and $\rpp\geq0$ on $(\rstar,1]$ provided the interval is nonempty.

When $\sigma=0$, the boundary condition \eqref{BC1} simplifies to $\rpp(1)=0$, and so $\rstar=1$.

When $\sigma<0$, we have by Lemma~\ref{gammapos} that $\gamma\leq-a^2j_1''(a)/b^2i_1''(b)$ and so
\[
\rpp(1)=a^2j_1^{\prime\prime}(a)+\gamma b^2i_1^{\prime\prime}(b)\leq 0
\]
We then take $\rstar=1$ once again.

When $\sigma>0$, we have $\gamma\geq-a^2j_1''(a)/b^2i_1''(b)$, and so $\rpp(1)\geq0$ in this case. We then have some root $\rstar\in(0,1]$.

Now we consider $(\rho-r\rp)(r)$. Note that $(\rho-r\rp)(0)=0$ and \\ $\frac{d}{dr}(\rho-r\rp)=-r\rpp$. Thus the function $(\rho-r\rp)$ is increasing and hence positive on $[0,\rstar]$. On $(\rstar,1]$, $(\rho-r\rp)$ is decreasing, and so on this interval
\begin{align*}
 (\rho-r&\rp)\geq (\rho-r\rp)(r=1)=j_1(a)-aj_1'(a)+\gamma(i_1(b)-bi_1'(b))\\
 &=aj_2(a)-\gamma bi_2(b) \geq \frac{1}{b^2i_1(b)}\Big(aj_2(a)b^2i_1(b)-a^2j_1(a)bi_2(b)\Big) \\
 &\qquad\qquad\qquad\qquad\qquad\text{by Lemmas~\ref{Besselprops} and~\ref{gammapos}}\\
 &=\frac{a^2(j_1(a)+j_3(a))b^2i_1(b)-a^2j_1(a)b^2(i_1(b)-i_3(b))}{(d+2)b^2i_1(b)}\\
 &\qquad\qquad\qquad\qquad\qquad\text{by Lemma~\ref{Besselprops}}\\
 &=\frac{a^2j_3(a)b^2i_1(b)+a^2j_1(a)b^2i_3(b)}{(d+2)b^2i_1(b)},
\end{align*}
which is positive by Lemma~\ref{Besselsigns}. This $\rho-r\rp\geq0$ on $[0,1]$ as desired.

Next we consider the function $\delr$. Like $\rpp$, this function is also convex on $[0,1]$, since the second derivative in $r$ is equal to
\[
 \frac{d^2}{dr^2}\delr(r)=-a^4j_1^{\prime\prime}(ar)+\gamma b^4i_1^{\prime\prime}(br)
\]
and is positive by Lemmas~\ref{Besselsigns} and~\ref{gammapos}. Note that $\delr(0)=-a^2j_1(0)+\gamma b^2i_1(0)=0$, and recall that we can write 
\[
\delr = \rpp-\frac{d-1}{r^2}(\rho-r\rp).
\]
We then must have $\delr\leq\rpp$ on $[0,1]$, and hence $\delr\leq-$ on $[0,\rstar]$. 

By the boundary condition \eqref{BC1}, we have $(1-\sigma)\rpp+\sigma(\delr)=0$ at $r=1$, so either $\rpp(1)=\delr(1)=0$ or $\rpp(1)$ and $\delr(1)$ have opposite signs. If both are zero, then we have $\rstar=1$, and $\delr\leq\rpp\leq0$ on $[0,1]$ as desired.

Suppose $\rpp(1)\neq0$. Then because $\delr\leq\rpp$ and $\delr(1)$ and $\rpp(1)$ have opposite signs, we must have $\rpp(1)>0$, and so $\delr(1)<0$. By convexity of $\delr(r)$, we then conclude $\delr(r)\leq0$ on $[0,1]$.

Finally, returning to considering all $\sigma\in[-1/(d-1),1]$, we investigate the sign of $\tau\rp-(\delr)_r$. Differentiating, we see
\[
 \frac{d}{dr}\Big(\tau\rp-(\delr)_r\Big)=\tau\rpp-(\delr)_{rr}=a^2b^2(\jpp(a)-\gamma\,\ipp(b))
\]
which is negative by sign properties of $\jpp$, $\ipp$, and $\gamma$. So the function\\ $\tau\rp-(\delr)_r$ will be minimal when $r=1$. However, at $r=1$, we may apply the boundary condition $Vu\Big|_{r=1}=0$:
\begin{align*}
 Vu&=\tau\rp-(\delr)_r-\frac{(1-\sigma)(d-1)}{r^3}(\rho-r\rp)=0\\
 &\text{and so}\qquad\tau\rp-(\delr)_r=\frac{(1-\sigma)(d-1)}{r^3}(\rho-r\rp).
\end{align*}
Since $\rho-r\rp\geq0$ on $[0,1]$, and $1-\sigma\geq0$ for all $\sigma$ under consideration we have that $\tau\rp-(\delr)_r\geq0$ at $r=1$ and hence on all of $[0,1]$.
\end{proof}

\section{ Proof of the isoperimetric inequality}\label{monotonesec}
In this section we establish the lemmas needed to prove the free plate isoperimetric inequality for nonzero $\sigma$. Some of the work from the proof of the inequality for the $\sigma=0$ case, found in \cite{inequalitypaper}, can be applied to our more general case of $\sigma\in[0,1)$.

The proof proceed as follows:
\begin{itemize}
 \item Definition of trial functions
 \item Evaluating the Rayleigh Quotient for these trial functions for regions $\Omega$ with volume equal to that of the unit ball
 \item Establishing partial monotonicity of the integrand in the numerator and denominator
 \item Proving the theorem using scaling and rearrangement arguments
\end{itemize}

We first bound our fundamental tone above by a quotient of integrals whose integrands are radial functions. The numerator will be quite complicated, so we write
\begin{align*}
 N[\rho]:&=(1-\sigma)\left((\rpp)^2+\frac{3(d-1)}{r^4}(\rho-r\rp)^2\right)\\
&\qquad+\sigma\left(\rpp-(d-1)\frac{\rho-r\rp}{r^2}\right)^2+\tau(\rp)^2+\frac{\tau(d-1)}{r^2}\rho^2
\end{align*}

We will also need the following calculus facts:
\begin{fact}\cite[Appendix]{cthesis}\label{derivs} We have the sums
\begin{align*}
&\sum_{k=1}^d|u_k|^2=\rho^2
&&\sum_{k=1}^d|D^2u_k|^2= (\rpp)^2+\frac{3(d-1)}{r^4}(\rho-r\rp)^2\\
&\sum_{k=1}^d|Du_k|^2=\frac{d-1}{r^2}\rho^2+(\rp)^2
&&\sum_{k=1}^d(\Delta u_k)^2 =(\delr)^2.
\end{align*}
\end{fact}

We may now use the trial functions to bound our fundamental tone by a quotient of integrals.
\begin{lemma}\label{lemmaboundRC} (Using the trial functions)
 For any $\Omega$, translated as in\\ Lemma~\ref{trialfcn}, we have
\begin{equation}
\omega \leq \frac{\int_\Omega N[\rho]\,dx}{\int_\Omega \rho^2\,dx}\label{boundRC}
\end{equation}
with equality if $\Omega=\Ostar$.
\end{lemma}

\begin{proof}
For $u_k$ defined as in Lemma~\ref{trialfcn}, we have
\[
\omega \leq Q[u_k] = \frac{\int_\Omega(1-\sigma)|D^2u_k|^2+\sigma(\Delta u_k)^2\tau|Du_k|^2\,dx}{\int_\Omega|u_k|^2\,dx},
\]
from the Rayleigh-Ritz characterization. We have equality when $\Omega=\Ostar$ because the $u_k$ are the eigenfunctions for the ball associated with the fundamental tone, by our choice of trial functions and hypothesis in Theorem~\ref{mainthm}.  Multiplying both sides by $\int_\Omega|u_k|^2\,dx$ and summing over all $k$, we obtain
\begin{equation}
\omega \int_\Omega \sum_{k=1}^d|u_k|^2\,dx
\leq \int_\Omega(1-\sigma)\sum_{k=1}^d|D^2u_k|^2+\sigma\sum_{k=1}^d(\Delta u_k)^2+\tau\sum_{k=1}^d|Du_k|^2\,dx\label{multsum}
\end{equation}
again with equality if $\Omega=\Ostar$.

By these sums in Fact~\ref{derivs}, we see inequality \eqref{multsum} becomes
\begin{align*}
\omega \int_\Omega& \rho^2\,dx\\
&\leq (1-\sigma) \int_\Omega\left((\rpp)^2+\frac{3(d-1)}{r^4}(\rho-r\rp)^2+\tau(\rp)^2+\frac{\tau(d-1)}{r^2}\rho^2\right)\,dx\\
&\qquad+\sigma\int_\Omega \left(\left((d-1)\frac{\rho-r\rp}{r^2}-\rpp\right)^2+\tau(\rp)^2+\frac{\tau(d-1)}{r^2}\rho^2\right)\,dx,
\end{align*}
once more with equality if $\Omega$ is the ball $\Ostar$. Dividing both sides by $\int_\Omega \rho^2\,dx$, we obtain \eqref{boundRC}.
\end{proof}

We now wish to show the quotient \eqref{boundRC} in Lemma~\ref{lemmaboundRC} has a sort of monotonicity with respect to the region $\Omega$, and so we examine the integrands of the numerator and denominator separately. The case of the denominator is much simpler; the partial monotonicity of the integrand of the numerator is much more difficult, and requires several lemmas.

We begin with the denominator. 
\begin{lemma}\label{mondenom}(Monotonicity in the denominator)
 The function $\rho(r)^2$ is increasing.
\end{lemma}
\begin{proof}
Differentiating, we see
\[
 \rp(r)=\begin{cases} j_1^\prime(ar)+\gamma i_1^\prime(br) &\text{when $0\leq r \leq 1$,}\\
         R^\prime(1) &\text{when $r\geq1$.}
        \end{cases}
\]
Obviously $i_1^\prime(br)\geq 0$. Because we have $a<\pll$ from the proof of Proposition \ref{fundmode}, the function $j_1^\prime(ar)$ is positive on $[0,1]$. Thus $\rp(r)$ is positive everywhere, and $\rho$ (and therefore $\rho^2$) is an increasing function.
\end{proof}

We do not need to prove the integrand of the numerator is strictly decreasing; a weaker ``partial monotonicity" condition is sufficient. We will say a function $F$ is \emph{partially monotonic for $\Omega$} if it satisfies
\begin{equation}
F(x)> F(y) \qquad\text{for all $x\in\Omega$ and $y\not\in\Omega$.}\label{moncond}
\end{equation}

Our approach to proving partial monotonicity of the numerator will depend on the sign of $\sigma$. We will also now assume that $\Omega$ has volume equal to that of the \emph{unit} ball, so that $\Ostar=\BB(1)$; we will recover the general case by a scaling argument at the end of the proof.

\subsection{Positive $\sigma$}
When $\sigma>0$, we will wish to group terms in $N[\rho]$ in order to address them separately. So we write
\[
 N[\rho]=(1-\sigma)(\rpp)^2+(1-\sigma)h(r)+\sigma g(r),
\]
where we define
\[
 h(r):=\frac{3(d-1)}{r^4}(\rho-r\rp)^2+\tau\Big((\rp)^2+\frac{(d-1)}{r^2}\rho^2\Big)
\]
and 
\[
g(r):=(\delr)^2+\tau\Big((\rp)^2+\frac{(d-1)}{r^2}\rho^2\Big).
\]

\begin{lemma}\label{monnum}(Partial monotonicity in the numerator when $\sigma\geq0$)

Suppose $\sigma\in[0,1)$ and one of the following is true:
\begin{itemize}
 \item we have $d=2$ or $d=3$ and $\tau>0$, or
 \item we have $d\geq 4$ and $\tau\geq(d+2)/2>a^2$,
\end{itemize}
then the function
\[
N[\rho]=(1-\sigma)(\rpp)^2+(1-\sigma)h(r)+\sigma g(r)
\]
satisfies the partial monotonicity condition \eqref{moncond} for the unit ball.
\end{lemma}

\begin{proof}
We consider each term of $N[\rho]$ separately.

Because $\rho$ is linear for $r>1$, we have $\rpp=0$ for $r> 1$. So the function $(\rpp)^2$ is nonnegative on $[0,1)$ and zero otherwise, and hence satisfies condition \eqref{moncond} for the unit ball. Since $\sigma<1$, we conclude that $(1-\sigma)(\rpp)^2$ satisfies \eqref{moncond}.

We are considering only $\sigma\in[0,1)$, so we will have partial monotonicity in the other terms if we can show that $h$ and $g$ are also decreasing functions of $r$. This is established in Lemmas~\ref{gdec} and~\ref{hdec}, below.
\end{proof}

\begin{rmk} The requirement that $\tau\geq (d+2)/2$ when dimension $d\geq4$ comes from observing that solving Lemma~\ref{atbounds} for $a^2$ gives us the bound
\[
 a^2\leq\frac{-t+\sqrt{t^2+4(d+2)\tau}}{2},
\]
and so $\tau>a^2$ is certainly true when $\tau$ exceeds this upper bound on $a^2$; this occurs when $\tau\geq (d+2)/2$.
\end{rmk}

\begin{lemma}\label{gdec} For dimensions $d=2,3$, when $\tau>0$ the function $g(r)$ is decreasing for $r\in(0,1)$. 

For all dimensions $d\geq 4$, when $\tau\geq a^2$, the function $g(r)$ is decreasing for $r\in(0,1)$.
\end{lemma}

Numerical computations in Mathematica strongly suggest that $g(r)$ is decreasing in $r$ for any choice of dimension and all positive $\tau$. However, our method of proof in the case of $\tau<(d+2)/2$ relies on an upper bound of the Bessel function $\ipp(z)$ on the interval $[0,d+2]$, and this bound is increasingly poor for high dimensions and proves to be too large for dimensions $d\geq 4$. Hence we restrict ourselves to small dimensions.

\begin{proof}
 We compute $g'(r)$ directly and simplify using the relationship between $\delr$ and $\rpp$:
\begin{align*}
 g'(&r)=2(\delr)(\delr)_r+2\tau\Big(\rp\rpp-\frac{(d-1)}{r^3}\rho(\rho-r\rp)\Big)\\
&=2(\delr)(\delr)_r+2\tau\rp\Big(\delr+\frac{(d-1)}{r^2}(\rho-r\rp)\Big)-\frac{2\tau(d-1)}{r^3}\rho(\rho-r\rp)\\
&=2(\delr)\Big((\delr)_r+\tau\rp\Big)-2\tau\frac{(d-1)}{r^3}(\rho-r\rp)^2.
\end{align*}
The second term in the final line is clearly negative when $r>0$. We have from Lemma~\ref{rhoproperties} that $\delr\leq0$, so it remains only to prove that
\[
 (\delr)_r+\tau\rp \geq 0 \qquad\text{on $[0,1)$}.
\]
On this interval, we can write $\rho$ as a linear combination of Bessel functions, and so we have
\begin{align*}
 (\delr)_r+\tau\rp 
&=a(\tau-a^2)j_1^\prime(ar)+\gamma b(\tau+b^2)i_1^\prime(br). 
\end{align*}
Since $a<\pll$, this is clearly positive for all $\tau\geq a^2$. 

Now suppose $\tau<a^2$. Then $\tau-a^2$ is negative, and by bounds \ref{bblem} on Bessel functions, we have
\[
 a(\tau-a^2)j_1^\prime(ar)+\gamma b(\tau+b^2)i_1^\prime(br)\geq a(\tau-a^2)c_{1,0}+\gamma b(\tau+b^2)c_{1,0}.
\]
So it suffices to show
\[
a(\tau-a^2)+\gamma b(\tau+b^2)\geq0,
\]
or equivalently,
\begin{equation}\label{smalltpossig}
\gamma\geq\frac{a(a^2-\tau)}{b(b^2+\tau)}=:\gstar. 
\end{equation}
To prove this, we will obtain a rational lower bound on gamma, and show it remains greater than $\gstar$ for all values of $a$ and $\tau$ under consideration, and the dimensions $d=2,3$.

Note that because $\tau\geq a^4/(d+2-a^2)$ by Lemma~\ref{atbounds}, we have in this case that
\[
 a^2>\frac{a^4}{d+2-a^2}\qquad\Rightarrow\qquad a^2<\frac{d+2}{2},
\]
and since $b^2=\tau+a^2$, we see
\[
 b^2\leq 2a^2=d+2=:M.
\]
Recall that by Lemma~\ref{gammapos} for nonnegative $\sigma$ we have the lower bound\\ $\gamma\geq -a^2\jpp(a)/b^2\ipp(b)$. We may now use our bounds on $-\jpp$ and $\ipp$ from Lemma~\ref{bblem}; then for any $a^2\leq\sqrt{d+2}$ and $b^2\leq M$, we have:
\[
 \gamma\geq \frac{a^3 d_1-a^5 d_2}{b^3 d_1+ kd_2 b^5},
\]
where $k=7/5+8(e^{M/4}-1)/5M$. We took $M=d+2$, so we will treat $k$ as a function of dimension $d$. Note also that $d_1/d_2=6(d+4)/5$ depends only on $d$. Thus to satisfy~\eqref{smalltpossig}, it suffices to show
\[
 \frac{6(d+4)a^2-5a^4}{6(d+4)b^2 + 5k b^4}-\frac{a^2-\tau}{b^2+\tau}\geq0.
\]
Since both denominators are positive, this is equivalent to proving
\[
 (6(d+4)a^2-5a^4)(b^2+\tau)-(a^2-\tau)(6(d+4)b^2 + 5kd_2 b^4)\geq0.
\]
The left-hand side is a polynomial in $\tau$ and $a^2$ (recall that $b^2=a^2+\tau$). Writing $x=a^2$, we define
\begin{align*}
 f(x,&\tau):=(6(d+4)x-5x^2)(x+2\tau)-(x-\tau)(6(d+4)(x+\tau) + 5k (x+\tau)^2)\\
 &=5k\tau^3+(6(d+4)+5kx)\tau^2+(12(d+4)-5x(k+2))x\tau-5(1+k)x^3.
\end{align*}
Thus proving nonnegativity of $f$ for $0<\tau<x<(d+2)/2$ is sufficient to establish \eqref{smalltpossig} in this case and complete our proof.

First note that in $f$, the coefficients of $\tau^3$ and $\tau^2$ are positive for all $d$, $x$, and $k$ under consideration.  We wish to see when the coefficient of $\tau$ is positive for our values of $x$ under consideration. Since $x\leq (d+2)/2$, taking $k$ as defined above, we have
\begin{align*}
 12(d+4)-5x(k+2)&
 =3.5d+35-4e^{(d+2)/4}.
\end{align*}
By direct numerical computation, we see that this last expression is positive for $d=2,\dots,9$ and $x\in[0,(d+2)/2]$, and so the coefficient of $\tau$ in $f(x,\tau)$ is positive in this case. Hence $f$ is increasing in $\tau$ for these values of $x$ and these dimensions, although we only needed this for $d=2,3$. 

For $d=2$, the constant $k<2.09$, and so we have for $x\in[0,2]$ that
\begin{align*}
 f(x,\tau)&\geq f\left(x,\frac{x^2}{4-x}\right)
 =\frac{x^3}{5(4-x)^3}\Big(816 - 88 x + 280 x^2 - 25 x^3\Big)\\
 &\geq\frac{x^3}{5(4-x)^3}\Big(230x^2+640\Big)\geq0
\end{align*}
with this last by noting that since $0\leq x\leq 2$, we have $-88x\geq-176$ and $-25x^3\geq-50x^2$.

For $d=3$, our constant $k<2.2$, and so for $x\in[0,5/2]$ we have
\begin{align*}
 f(x,\tau)&\geq f\left(x,\frac{x^2}{5-x}\right)=\frac{x^3}{(5-x)^3}\Big(100 + 45 x + 67 x^2 - 5 x^3\Big)\\
 &\geq\frac{x^3}{5(4-x)^3}\Big(\frac{109}{2}x^2+45x+100\Big)\geq0
\end{align*}
by noting $-5x^3\geq-25x^2/2$.
\end{proof}

\begin{rmk} The function $f(x,\tau)$ in the above proof will \emph{not} be nonnegative for all $x\in[0,(d+2)/2]$ when our dimension $d\geq 4$; our bounds for $\ipp(z)$ on $[0,d+2]$ are too large since $k$ grows exponentially in dimension. Numerical investigations support the conjecture that we still have $\gamma\geq\gstar$ for small $\tau$ in higher dimensions, but we would need need a better lower bound on $\gamma$ in order to prove this.
\end{rmk}

\begin{lemma}\label{hdec} For all dimensions $d\geq2$ and values $\sigma>0$ and $\tau>0$, the function $h(r)$ is decreasing on $[0,1]$. 
\end{lemma}

\begin{proof}
Recall $h(r)=\frac{3(d-1)}{r^4}(\rho-r\rp)^2+\tau\Big((\rp)^2+\frac{(d-1)}{r^2}\rho^2\Big)$.
 
 Consider $h'(r)$:
\begin{align}
 h'(r)&=\frac{-12(d-1)}{r^5}(\rho-r\rp)^2-\frac{6(d-1)}{r^3}(\rho-r\rp)\rpp\nonumber\\
&\qquad-\frac{2\tau(d-1)}{r^3}(\rho-r\rp)\rho+2\tau\rp\rpp. \nonumber\\
&=\frac{-2(d-1)}{r^3}(\rho-r\rp)\left(\frac{6}{r^2}(\rho-r\rp)+3\rpp+\tau\rho\right)+2\tau\rp\rpp \label{negrpp}
\end{align}
Writing $(d-1)(\rho-r\rp)/r^2=\rpp-\delr$, we can also rewrite this as
\begin{align}
&=\frac{-2}{r}(\rpp-\delr)\left(\frac{6}{r^2}(\rho-r\rp)+3\rpp+\tau\rho\right)+2\tau\rp\rpp \nonumber\\
&=\frac{2}{r}\delr\left(\frac{6}{r^2}(\rho-r\rp)+3\rpp+\tau\rho\right)\nonumber\\
&\qquad+\frac{-2}{r}\rpp\left(\frac{6}{r^2}(\rho-r\rp)+3\rpp+\tau(\rho-r\rp)\right) \label{posrpp}.
\end{align}

We know from Lemma~\ref{rhoproperties} that $\rpp<0$ on $(0,\rstar)$ and $\rpp>0$ on $(\rstar,1]$. We consider each case separately.

When $r\in(0,\rstar)$, we write $h'(r)$ as in \eqref{negrpp}. That $h'(r)<0$ in this case follows from the proof of the free plate isoperimetric inequality for $\sigma=0$ in \cite[Lemmas~18 through 22]{inequalitypaper}. These lemmas rely on properties of ultraspherical Bessel functions from \cite{besselpaper} and the following properties of the function $\rho$:
\begin{enumerate}
\item $\rho(r)=j_1(ar)+\gamma i_1(br)$ with $b=\sqrt{a^2+\tau}$, $0<a<\pll$, and $\gamma$ determined by the natural boundary conditions.
\item $(d+2)\tau>\omega^*>\tau d$.
\item $\rpp\leq 0$ for all $r$ under consideration
\item $\gamma\geq-a^2j_l^{\prime\prime}(a)/b^2i_1^{\prime\prime}(b)$ (the proof for $\sigma =0$ assumes equality and establishes a lower bound)
\item $\rho-r\rp\geq 0$ for all $r$ under consideration
\end{enumerate}

Because $\rstar\leq 1$, we meet condition (1) by our choice of trial functions. The bound on $\omega^*$, (2), is guaranteed by Lemma~\ref{wbounds} and Proposition~\ref{propLS}.  The bounds (3), (4), and (5) hold on $[0,\rstar]$ by Lemmas~\ref{gammapos} and~\ref{rhoproperties}. Thus we have met all the hypotheses of the lemmas from \cite{inequalitypaper}, and so $h'(r)\leq0$ for $r\in[0,\rstar]$ as desired.

When $r\in(\rstar,1]$, we write $h'(r)$ as in \eqref{posrpp}. We know $\rho-r\rp\geq0$ and $\delr\leq0$ here while $\rpp\geq0$, so both terms in \eqref{posrpp} are negative. Thus $h'(r)\leq0$ for all $r\in(0,1)$, completing the proof.
\end{proof}

\subsection{ Negative $\sigma$ (Auxetic case)}
Throughout this section we will use the notation $\alpha=|\sigma|$ to reduce risk of confusion over signs; note now that the range of values we consider is $0<\alpha<1/(d-1)$. 

Again, we will treat large and small values of $\tau$ separately. For this section, the ``small'' values of $\tau$ will be any which satisfy the inequality
\[
0<\tau\leq\frac{3(1+\alpha)a^2}{(d+2)(1-\alpha)}=:\tmax.
\]
Let us collect some results on bounds of $a$ and $b$ when $\tau$ is small.
\begin{lemma}\label{tmaxprops} Suppose $0<\tau\leq\tmax$. Then we have
\[
 a^2\leq\frac{3(1+\alpha)(d+2)}{d+5-\alpha(d-1)}:=\xmax \quad\text{and}\quad b^2\leq\frac{3(1+\alpha)}{1-\alpha}=:\bmaxsq.
\]
In particular, when $d=3$, we have $\tmax\leq 9a^2/5$ and $\xmax\leq45/14$. When our dimension $d\geq 4$, we have $\bmaxsq\leq 6$.
\end{lemma}

\begin{proof}
First, we will use $\tau\leq\tmax$ to restrict the values of $a$ we need to consider. Since we have that $\tau>a^4/(d+2-a^2)$ by Lemma~\ref{atbounds}, the values of $a$ must satisfy
\[
 \frac{3(1+\alpha)a^2}{(d+2)(1-\alpha)}-\frac{a^4}{d+2-a^2}\geq0.
\]
Solving this inequality for $a^2$ gives us the desired bound $a^2\leq\xmax$.

Write $\xmax=\xmax(d,\alpha)$. By inspection, $\xmax$ is increasing in $\alpha$, and so for any fixed $d$ is maximized when $\alpha=1/(d-1)$. Evaluating $\xmax(d,\alpha)$ at this value and  differentiating formally then yields
\[
 \frac{\partial}{\partial d}\xmax(d,1/(d-1))= \frac{\partial}{\partial d}\left(\frac{3d(d+2)}{(d-1)(d-4)}\right)=\frac{3(d^2-8d-8)}{(d-1)^2(d+4)^2}.
\]
Due to the quadratic term in the numerator, this derivative is positive when $d\geq 9$ and negative when $2\leq d\leq 8$. The limit as $d\to\infty$ of $\xmax(d,1/(d-1))$ is $3$ by inspection. Thus by simple calculus we obtain 
\[
 \xmax(d,\alpha)\leq\frac{3d(d+2)}{(d-1)(d-4)}\leq\max\{\xmax(2,1),3\}=\frac{45}{14}.
 \]
We can also use the bounds on $\tau$ and $a^2$ to find an upper bound on $b^2$:
\begin{align*}
 b^2&=a^2+\tau\leq a^2\left(1+\frac{3(1+\alpha)}{(d+2)(1-\alpha)}\right)\\
 &\leq \xmax(d,\alpha)\left(1+\frac{3(1+\alpha)}{(d+2)(1-\alpha)}\right)=\frac{3(1+\alpha)}{1-\alpha}=:\bmaxsq.
\end{align*}
By inspection, this upper bound $\bmaxsq$ is increasing in $\alpha$; if we take $d\geq4$ then $\alpha\leq 1/(d-1)$ and so in this case our bound becomes 
\[ 
 b^2\leq\bmaxsq\leq\frac{3d}{d-2}.
\]
The right-hand side is decreasing in $d$, so when $d\geq4$ we see $\bmaxsq\leq 6$.
\end{proof}

We are now ready to state our results for monotonicity of the Rayleigh quotient numerator for the auxetic plate:
\begin{prop}\label{auxeticparmon} Suppose one of the following holds:
\begin{enumerate}
 \item The dimension $d\geq3$ with any $0<\alpha<1/(d-1)$ and any $\tau>0$
 \item The dimension $d=2$ with $\tau>\frac{3(1+\alpha)a^2}{(d+2)(1-\alpha)}$ and $0<\alpha<1$.
 \item The dimension $d=2$ with $\tau<\frac{3(1+\alpha)a^2}{(d+2)(1-\alpha)}$ and $0<\alpha<51/97$.
\end{enumerate}
Then we have partial monotonicity on the interval $[0,1]$ of the Rayleigh quotient numerator
\begin{align*}
 N[\rho]&=(1+\alpha)\left((\rpp)^2+\frac{3(d-1)}{r^4}(\rho-r\rp)^2\right)\\
&-\alpha(\delr)^2+\tau(\rp)^2+\frac{\tau(d-1)}{r^2}\rho^2.
\end{align*} 
\end{prop}

\begin{proof} Because we have defined our trial function so that $\rpp=0$ for $r>1$, we already have partial monotonicity of $(\rpp)^2$. We then focus on the remaining terms:
\begin{align*}
 \tilde{N}(r)&=N[\rho]-(1+\alpha)(\rpp)^2\\
&=(1+\alpha)\frac{3(d-1)}{r^4}(\rho-r\rp)^2-\alpha(\delr)^2+\tau(\rp)^2+\frac{\tau(d-1)}{r^2}\rho^2.
\end{align*}
Differentiating and regrouping the above yields
\begin{align*}
 \frac{1}{2}\tilde{N}'(r)&=-\frac{(d-1)}{r^3}(\rho-r\rp)\left(\frac{6(1+\alpha)}{r^2}(\rho-r\rp)+3(1+\alpha)\rpp+\tau\rho\right)\\
 &\qquad+\tau\rp\rpp-\alpha(\delr)(\delr)_r
\end{align*}
We will want to handle the term $\tau\rp\rpp-\alpha(\delr)(\delr)_r$ differently depending on our choice of $\tau$.

We will consider the ``large'' $\tau$ case first. For this we rewrite the $\tau\rp\rpp$ term using $\rpp=\delr+\frac{d-1}{r}(\rho-r\rp)$. We then obtain:
\begin{align*}
  \frac{1}{2}&\tilde{N}'(r)\\
 &=-\frac{(d-1)}{r^3}(\rho-r\rp)\left(\frac{6(1+\alpha)}{r^2}(\rho-r\rp)+3(1+\alpha)\rpp+\tau\rho-\alpha\tau r\rp\right)\\
 &\qquad+(1-\alpha)\tau\rp\rpp+\alpha\delr\Big(\tau\rp-(\delr)_r\Big).
\end{align*}
Consider the second line, which we will denote by $k(r)$:
\[
 k(r)=(1-\alpha)\tau\rp\rpp+\alpha\delr\Big(\tau\rp-(\delr)_r\Big)
\]
Because $\rpp\leq0$ and the parameter $0<\alpha\leq1$, the first term in $k(r)$ is nonpositive. The second term will be nonpositive for any values of $r$ such that $(\delr)_r\leq0$, since $\rp\geq0$ and $\delr\leq0$ for all $r\in[0,1]$.

We will thus assume $(\delr)_r>0$ and show that $k(r)\leq0$ for these values of $r$. In this case, the sign of the second term in $k(r)$ depends on the sign of $\tau\rp-(\delr)_r$. By Lemma~\ref{rhoproperties} this is nonnegative, and so $k(r)\leq0$ as desired.

Let us now consider the other term of $\frac{1}{2}\tilde{N}'(r)$:
\[
 -\frac{(d-1)}{r^3}(\rho-r\rp)\left(\frac{6(1+\alpha)}{r^2}(\rho-r\rp)+3(1+\alpha)\rpp+\tau\rho-\alpha\tau r\rp\right).
\]
By positivity of $\rho-r\rp$ on $[0,1]$, the above will be negative if the final factor is negative. We will denote this term by $l(r)$:
\[
 l(r)=\frac{6(1+\alpha)}{r^2}(\rho-r\rp)+3(1+\alpha)\rpp+\tau\rho-\alpha\tau r\rp.
 \]
As in the $\sigma>0$ case, we write $\rpp$ in terms of $\rho,\rp,\delr$ and split $\tau\rho$ into two pieces:
\[
 l(r)=\left(\frac{3(d+1)(1+\alpha)}{r^2}+\alpha\tau\right) (\rho-r\rp)
 +3(1+\alpha)\delr+(1-\alpha)\tau\rho.
\]
Then expressing $\rho$ and its derivatives in terms of $j_1, i_1, j_3$ and $i_3$ using properties of Bessel functions, we obtain:
\begin{align*}
 l(r)&=\left((1-\alpha)\tau-\frac{3(1+\alpha)a^2}{d+2}\right)j_1(ar)+\gamma \left((1-\alpha)\tau+\frac{3(1+\alpha)b^2}{d+2}\right)i_1(br)\\
 &\qquad+\left(\frac{3(d+1)(1+\alpha)}{d+2}\right)\Big(a^2j_3(ar)+\gamma b^2 i_3(br)\Big)+\alpha\tau(\rho-r\rp).
\end{align*}
We now have four terms to consider. Because $\rho-r\rp\geq0$ for all $r\in[0,1]$, our parameter $0<\alpha\leq1$ regardless of dimension, and properties of $a$ and $j_3$ and $i_3$, both terms in the last line are nonnegative. By inspection, the coefficient of the $i_1(br)$ term will be positive. 

Thus the sign of $l(r)$ (and hence partial monotonicity of $N[\rho]$) hinges on the term involving $j_1(ar)$. When this term is positive, all terms in $l(r)$ are nonnegative and so $l(r)\geq0$ on $[0,1]$. The $j_1(a,r)$ term's coefficient is positive when
\[
\tau\geq\frac{3(1+\alpha)a^2}{(d+2)(1-\alpha)}=\tmax,
\]
which is precisely our ``large'' $\tau$ condition. Thus we've shown $l(r)\geq 0$ for these values of $\tau$, and hence $N[\rho]$ has the desired partial monotonicity.

If $\tau<\tmax$ and $d\geq 4$, then by Lemma~\ref{smalltd4} we again have that $l(r)\geq0$.

This leaves the $\tau<\tmax$ case for dimensions $d=2,3$. In this case, one can choose $a$, $\tau$, and $r$ so that the function $l(r)$ will be negative, and so we need to change how we group the terms in $\tilde{N}'(r)$ in order to achieve partial monotonicity. This time we'll rewrite the $-\alpha\delr(\delr)_r$ term:
\begin{align*}
 \tilde{N}'&(r)=-\frac{(d-1)}{r^3}(\rho-r\rp)\left(\frac{6(1+\alpha)}{r^2}(\rho-r\rp)+3(1+\alpha)\rpp+\tau\rho\right)\\
 &\qquad+\tau\rp\rpp-\alpha\left(\rpp-\frac{d-1}{r^2}(\rho-r\rp)\right)(\delr)_r\nonumber\\
  &=-\frac{(d-1)}{r^3}(\rho-r\rp)\left(\frac{6(1+\alpha)}{r^2}(\rho-r\rp)+3(1+\alpha)\rpp+\tau\rho-\alpha r(\delr)_r\right)\\
 &\qquad+(1-\alpha)\tau\rp\rpp+\rpp\Big(\tau\rp-(\delr)_r\Big).
\end{align*}
As before, note that $(1-\alpha)\tau\rp\rpp$ and $\rpp\Big(\tau\rp-(\delr)_r\Big)$ will both be nonpositive. Thus it suffices to show the nonnegativity of
\[
\tilde{l}(r):=\frac{6(1+\alpha)}{r^2}(\rho-r\rp)+3(1+\alpha)\rpp+\tau\rho-\alpha r(\delr)_r.
\]
If $\tau<\tmax$ and $d=3$ with $0<\alpha<1/(d-1)$ or $d=2$ with $\alpha<51/97$, then by Lemma~\ref{smallt23} this is nonnegative for any $r\in[0,1]$, completing our proof.\end{proof}

\begin{lemma}\label{smalltd4} For dimensions $d\geq 4$ and all $0<\alpha<1/(d-1)$, if $\tau$ and $a$ are such that
\[
 \tau<\frac{3(1+\alpha)a^2}{1-\alpha}, 
\]
then the function
\[
 \tilde{l}(r)=\frac{6(1+\alpha)}{r^2}(\rho-r\rp)+3(1+\alpha)\rpp+\tau\rho-\alpha\tau r\rp
 \]
is nonnegative for all $r\in[0,1]$.
\end{lemma}
\begin{proof}
We will find it useful to rewrite $\tilde{l}(r)$ in terms of $j_1$, $j_3$, $i_1$, and $i_3$ as we did above. Then for these small $\tau$ values, the coefficient of $j_1(ar)$ in $l(r)$ is negative while all others are positive, so by Lemma~\ref{bblem} and our work in the early part of the proof of Proposition~\ref{auxeticparmon}, we have 
\begin{align*}
\tilde{l}(r)&\geq \left((1-\alpha)\tau-\frac{3(1+\alpha)a^2}{d+2}\right)j_1(ar)+\gamma \left((1-\alpha)\tau+\frac{3(1+\alpha)b^2}{d+2}\right)i_1(br)\\
&\geq \left((1-\alpha)\tau-\frac{3(1+\alpha)a^2}{d+2}\right)c_0ar+\gamma \left((1-\alpha)\tau+\frac{3(1+\alpha)b^2}{d+2}\right)c_0br.
\end{align*}
Thus for small $\tau$ values and dimensions $d\geq 4$, it suffices to show that
\begin{equation}\label{smallt}
\left((1-\alpha)\tau-\frac{3(1+\alpha)a^2}{d+2}\right)a+\gamma \left((1-\alpha)\tau+\frac{3(1+\alpha)b^2}{d+2}\right)b\geq 0.
\end{equation}
By solving this inequality for $\gamma$, we see that the above inequality holds if and only if
\[
 \gamma\geq \frac{a}{b}\frac{3a^2(1+\alpha)-(d+2)(1-\alpha)\tau}{3b^2(1+\alpha)+(d+2)(1-\alpha)\tau}=:\gstar.
\]
Note that if we view this lower bound $\gstar$ as a function of $\alpha$ with the variables $a$, $\tau$, and $d$ seen as independent, then differentiating formally yields
\[
 \frac{d}{d\alpha}\gstar=\frac{a}{b}\frac{6(d+2)(a^2+b^2)\tau}{(3b^2(1+\alpha)+(d+2)(1-\alpha)\tau)^2}\geq 0.
\]
Thus, holding all other variables constant, $\gstar$ is increasing in $\alpha$. Recall from Lemma~\ref{gammapos} that similarly, for fixed $\tau,d,a$, we have that $\gamma$ is increasing in $\sigma$, and hence decreasing in $\alpha$. Then $\gamma-\gstar$ may be viewed as a decreasing function of $\alpha$ and so 
\[
 \gamma-\gstar\geq \gamma\left(\frac{1}{d-1}\right)-\gstar\left(\frac{1}{d-1}\right).
\]
Thus it suffices to prove the inequality in the extreme case $\alpha=1/(d-1)$. For this value of $\alpha$, we have
\[
 \gstar=\frac{a}{b}\cdot\frac{3a^2d-(d+2)(d-2)\tau}{3b^2d+(d+2)(d-2)\tau}.
\]
Using Bessel identities from Lemma~\ref{Besselprops}, in the $\alpha=1/(d-1)$ case we find
\[
 \gamma=\frac{-da^2\jpp(a)-a^2j_1(a)}{db^2\ipp(b)-b^2i_1(b)}=\frac{a^2j_2'(a)}{b^2i_2'(b)}.
\]
We wish to establish a lower bound on our $\gamma$ that is a rational function of $a$ and $b$, using our results from Lemma~\ref{bblem}. The upper bound $M=\bmaxsq=3d/(d-2)$ of $b^2$ is easily seen from Lemma~\ref{tmaxprops} to be decreasing as a function of $d$, while the constant $K$ is increasing in $M$. Since we only want dimensions $d\geq 4$ we may take the value of $M$ when $d=4$, yielding $K<5/3$.

Thus for the values of $\tau$, $a$ and $b$ under consideration, with $n_0$ and $n_1$ as defined in Lemma~\ref{bblem}, we have:
\[
 \gamma\geq \frac{n_0a^3-n_1a^5}{n_0b^3+\frac{5}{3}n_1b^5}=\frac{a^3((d+4)-a^2)}{b^3\left((d+4)+\frac{5}{3}b^2\right)}
 \]
 by factoring out $n_1$, since $n_0/n_1=d+4$. We now have that
\[
 \gamma-\gstar \geq \frac{a}{b}\left(\frac{a^2((d+4)-a^2)}{b^2\left((d+4)+\frac{5}{3}b^2\right)}-\frac{3da^2-(d^2-4)\tau}{3db^2+(d^2-4)\tau}\right).
 \]
Rewriting the right-hand side as a single quotient with positive denominator, we see that $\gamma-\gstar\geq0$ whenever the following expression is nonnegative:
\begin{equation}\label{num1}
a^2\Big((d+4)-a^2\Big)(3db^2+(d^2-4)\tau)-b^2\left((d+4)+\frac{5}{3}b^2\right)(3da^2-(d^2-4)\tau).
\end{equation}
Since $b^2=a^2+\tau$ and all powers of $a$ appearing in \eqref{num1} are even, we can set $x=a^2$ and rewrite the above as a polynomial $P$ in $x$ and $\tau$:
\begin{align*}
 p(x,\tau)&:=\frac{5}{3}(d^2-4)\tau^3+\frac{1}{3}\Big(5(2d^2-3d-8)x+3(d^2-4)(d+4)\Big)\tau^2\\
  &\qquad+\frac{1}{3}\Big((2d^2-39d-8)x+6(d^2-4)(d+4)\Big)x\tau -8dx^3
\end{align*}
We will now show that $p(x,t)$ is positive for all dimensions $d\geq 4$ and all $\tau$ and $x$ such that
\[
 0\leq x\leq \frac{3d(d+2)}{(d+4)(d-1)}:=\xmax,
 \qquad\frac{x^2}{d+2-x}<\tau<\tmax.
\]
The coefficient of $\tau^3$ in $p$ is positive by inspection. Next we look at the coefficient of $\tau^2$. The quadratic $2d^2-3d-8$ has roots at $d\approx -1.4, 2.9$ and so is positive for $d\geq 4$. Because we only consider $x\geq0$, we may conclude the coefficient of $\tau^2$ is positive.

The coefficient of $\tau$ requires a little more work. Again, $x\geq0$, so if $2d^2-39d-8$ is positive, we're done. However, this quadratic has roots at $d\approx-0.2, 19.7$. So for $d\geq20$, we have positivity of this coefficient. If \\$4\leq d\leq 19$, however, then 
\begin{align*}
 (2d^2-39d&-8)x+6(d^2-4)(d+4)\\
 &\geq (2d^2-39d-8)\xmax+6(d^2-4)(d+4)\\
 &=\frac{3(d+2)(2d^4+12d^3-51d^2-72d+64)}{(d-1)(d+4)}.
\end{align*}
The quartic term has four real roots that may be numerically estimated by standard techniques (eg, by Newton's method), occurring at $d\approx-8.5$, $-1.6$, $0.6$ and $3.5$. Thus for $d\geq4$, the above expression is positive, and hence so is the coefficient of $\tau$ in $p$.

Since the coefficients of positive powers of $\tau$ are all positive, we may conclude $p(x,\tau)$ is increasing in $\tau$, and so minimized when $\tau$ is. Evaluating $p$ at our minimum value $\tau=x^2/(d+2-x)$ and simplifying, we can write
\begin{align*}
 &p\left(x,\frac{x^2}{d+2-x}\right)=\frac{(d+2)x^3}{3(d+2-x)^3}m(x), \qquad\text{where}\\
 &m(x):=-3(d+2)x^3+3(3d^2-d-16)x^2\\
 &\qquad\qquad-(d+2)(7d^2-15d-64)x+6(d+2)^3(d-4)
\end{align*}
The quadratic $7d^2-15d-64$ has two real roots at $d\approx-2.1, 4.3$ and so the linear term in $x$ has a negative coefficient for $d\geq 5$. In this case, we have
\begin{align*}
 m(x)&\geq 3\Big((d+2)\xmax+(3d^2-d-16)\Big)x^2\\
 &\qquad-(d+2)(7d^2-15d-64)\xmax+6(d+2)^3(d-4)\\
 &=3x^2\frac{3d^4+5d^3-31d^2-32d+64}{(d-1)(d+4)}\\&\qquad+\frac{3(d+2)^2(d^2-d-8)(2 d^2-3d-8)}{(d-1)(d+4)}.
\end{align*}
The coefficient of $x^2$ will be positive when the numerator is. Since the quartic $3d^4+5d^3-31d^2-32d+64$ has four real roots at $d\approx -3.2,-2.2, 1.1$ and $2.6$, this term is always positive for $d\geq 5$. The two quadratics in the numerator of the constant term have all real roots at $d\approx -2.4, -1.4, 2.9,$ and $3.4$, and so this term is also positive for $d\geq 5$. Thus $m(x)\geq0$ for $0\leq x\leq\xmax$ when $d\geq 5$.

If $d=4$, then $\xmax=3$ and so for $0\leq x\leq 3$, we have
\[
 m(x)=6x(-x^2+14x+12)\geq6x(-9+14x+12)>0.
\]

Thus we have show that for all $d\geq4$ and $\tau, x=a^2$ under consideration, $p(x,\tau)$ and hence $\gamma-\gstar$ are nonnegative.
\end{proof}

\begin{rmk} This method of proof cannot be extended to the physical case\\ $d=2,3$; numerical investigations show that there exist values of $\tau$, $a$,  $\alpha$, and $r$ in those dimensions for which $l(r)<0$. We will need to group the terms in $\tilde{N}'(r)$ differently to obtain a proof for these small dimensions and small $\tau$. This alternate grouping can also be used to prove the small $\tau$ case for higher dimensions, however it is more cumbersome.
\end{rmk}

\begin{lemma}[Small $\tau$ for $d=2,3$] \label{smallt23}Suppose $d=3$ and $0<\alpha<1/2$ or $d=2$ and $\alpha\leq51/97$, and $\tau<\tmax$. Then we have that
\begin{equation}\label{foo}
\frac{6(1+\alpha)}{r^2}(\rho-r\rp)+3(1+\alpha)\rpp+\tau\rho-\alpha r\delr(\delr)_r\geq0
\end{equation}
and hence $N[\rho]$ is partially monotone on $\Ostar$.
\end{lemma}

\begin{rmk} The upper bound on $\alpha\leq51/97$ in the $d=2$ case is specific to our method of proof and not a strict upper bound on $\alpha$ values for which the inequality~\eqref{foo} holds. Numerical estimates suggest inequality~\eqref{foo} holds for all $0<\alpha<1$. 
\end{rmk} 

\begin{proof} Using the notation of Lemma~\ref{tmaxprops}, since $\tau\leq\tmax$, we also have $0\leq a^2\leq\xmax$ and $0\leq b\leq\bmax$.

As in our previous proofs, we use Bessel identities to write $\rho-r\rho'$ and $\rho'$ in terms of Bessel $j_1, i_1, j_3$, and $i_3$. Then \eqref{foo} can be rewritten as
\begin{align*}
 &\left(\tau-\frac{3(1+\alpha)}{d+2}a^2+\frac{\alpha a^2}{d+2}(d+2-a^2r^2)\right)j_1(ar)\\
&\qquad+\gamma \left(\tau+\frac{3(1+\alpha)}{d+2}b^2-\frac{\alpha b^2}{d+2}(d+2+b^2r^2)\right)i_1(br)\\
 &\qquad+\frac{a^2}{d+2}\Big(3(1+\alpha)(d+1)-\alpha a^2r^2\Big)j_3(ar)\\
 &\qquad+\gamma\frac{b^2}{d+2}\Big(3(1+\alpha)(d+1)+\alpha b^2r^2\Big)i_3(br)\\
 &=:A j_1(ar)+\gamma Bi_1(br)+C j_3(ar)+\gamma Di_3(br)
\end{align*}
The coefficient $D$ of $i_3(br)$ is positive by inspection. Since $r\in[0,1]$ and\\ $a^2\in[0,d+2]$, the coefficient $C$ of $j_3(ar)$ can be bounded below as follows:
\[
 C\geq \frac{a^2}{d+2}\Big(3(1+\alpha)(d+1)-\alpha (d+2)\Big)=\frac{(3 + 3 d + \alpha + 2 d \alpha)a^2}{d+2}.
\]
Hence $C$ is positive. 

Let us look next at the coefficient $B$ of $i_1(br)$; our goal is to show that it is positive for those values of $b$ satisfying $b\leq\bmax$. We'll first rewrite $B$ as
\[
 B=\frac{(d+2)\tau+b^2(3-\alpha(d-1)-\alpha r^2b^2)}{d+2}.
\]
Note that $B$ is decreasing in both $r$ and $\alpha$. 

When $d=3$, we have $\alpha<1/2$, and so writing $b^2=a^2+\tau$ we obtain
\[
 B\geq \frac{1}{10}\Big(-\tau^2+2(7-a^2)\tau+4a^2-a^4\Big)=:B_3(\tau).
\]
The lower bound $B_3$ is concave in $\tau$, and so is minimized when $\tau=0$ or $\tau=\tmax$. Recall that for $d=3$ we have $\tmax\leq 9a^2/5$. Then evaluating, we see:
\begin{align*}
 B_3(0)&=\frac{1}{10}a^2(4-a^2)\geq\frac{1}{10}a^2(4-\xmax),\\
 B_3(\tmax)&=\frac{1}{125}a^2(365-98a^2)\geq\frac{1}{12250}a^2(3.72-\xmax),
\end{align*}
both of which are nonnegative, since $\xmax\leq45/14<3.22$.

When $d=2$, the coefficient $B$ is still decreasing in $r$ and $\alpha$. Since we are restricting ourselves to $\alpha\leq51/97$, this gives us
\[
 4B\geq -\frac{51}{97}\tau^2+\frac{628-102a^2}{97}\tau+\frac{a^2(80-17a^2)}{97}=:B_2(\tau).
\]
Since $B_2$ is convex in $\tau$, we look at $\tau=0$ and $\tau=\tmax\leq2.45a^2$:
\begin{align*}
 B_2(0)&\geq0.52(4.7-a^2)a^2\\
 B_2(\tmax)&\geq 6.12(2.04-a^2)a^2.
\end{align*}
When $\alpha\leq51/97$, we have $a^2\leq\xmax<2.84$; then both of these terms are positive.

Finally, we need to consider the sign of $A$. We'll write $A$ as
\[
 A=\frac{(d+2)\tau-a^2(3-\alpha(d-1)+\alpha a^2r)}{d+2}.
\]
Since $A$ is decreasing in $r$, we can bound it below by setting $r=1$. Note also that if $\tau$ is large enough relative to $a$, then $A$ will be nonnegative. However, if $\tau$ is too small, specifically if
\[
 \tau<\frac{(3-\alpha(d-1)+\alpha a^2)a^2}{d+2}=:\tmid,
\]
then we have $A<0$ for $r=1$.

For these values of $\tau$, we again apply our lower bound on $\tau$ from Lemma~\ref{atbounds} and solve for $a^2$, obtaining the bound
\[
 a^2\leq \frac{2d\alpha+\alpha-5-d+\sqrt{d^2+10d+25-2\alpha(2d^2+5d-7)+9\alpha^2}}{2\alpha}=:\xmid.
\]
Similarly, we set bound $\bmidsq:=\tmid+\xmid$, which gives us an upper bound on the values of $b^2$ for which $A$ is negative. Both $\xmid$ and $\bmidsq$ are extremely cumbersome to deal with in general dimension $d$; fortunately we only need to treat the cases $d=2$ and $d=3$. 

When $d=3$, our bounds $\xmid$ and $\bmidsq$ become
\[
 \xmid=\frac{7\alpha-8+\sqrt{64-\alpha(52-9\alpha)}}{2\alpha}\qquad\bmidsq=\frac{3\alpha-2+\sqrt{64-\alpha(52-9\alpha)}}{2}.
\]
If we differentiate $\bmidsq$ with respect to $\alpha$ and simplify by writing as a single quotient and rationalizing the numerator, we obtain
\[
 \frac{\partial\bmidsq}{\partial\alpha}=\frac{-100}{(2\sqrt{64-\alpha(52-9\alpha)})(3\sqrt{64-\alpha(52-9\alpha)}+(26-9\alpha))},
\]
which is negative. Thus $\bmidsq$ is maximal when $\alpha=0$, and so $\bmidsq\leq 3$.

When $d=2$, our bounds $\xmid$ and $\bmidsq$ become
\[
 \xmid=\frac{5\alpha-7+\sqrt{49-\alpha(22-9\alpha)}}{2\alpha}\quad\bmidsq=\frac{3\alpha-2+\sqrt{(49-\alpha(22-9\alpha)}}{2}.
\]
Again, we differentiate $\bmidsq$ with respect to $\alpha$ and simplify, finding
\[
 \frac{\partial\bmidsq}{\partial\alpha}=\frac{320}{(2\sqrt{49-\alpha(22-9\alpha)})(3\sqrt{49-\alpha(22-9\alpha)}+(11-9\alpha))}.
\]
We also do this with $\xmid$:
\[
 \frac{\partial\xmid}{\partial\alpha}=\frac{160}{\sqrt{49-\alpha(22-9\alpha)}(7\sqrt{49-\alpha(22-9\alpha)}+(49-11\alpha))}.
\]
So when our dimension $d=2$, we have that both $\xmid$ and $\bmidsq$ increase with $\alpha$, and so taking our largest value of $\alpha=51/97$, we see $\xmid\leq1.9$ and $\bmidsq\leq3.5$.

We now show that $l(r)$ is positive even when $\tau\leq\tmid$. Since we've already established the positivity of $C$ and $D$, we have that
\[
 l(r)\geq Aj_1(ar)+\gamma B i_1(br).
\]
Since we assume $\tau<\tmid$, we have $A<0<B$ and will apply the bounds from Lemma~\ref{bblem}
\[
 -j_1(z)\geq-c_{1,0}z\qquad\text i_1(z)\geq c_{1,0}z,
\]
where $c_{1,0}$ is a positive constant coming from the series expansion and depends only on the dimension. Then
\begin{align*}
 l(r)&\geq c_0r(Aa+\gamma Bb)\\
 &=c_0r\left(a\Big((d+2)\tau-a^2(3-\alpha(d-1)+\alpha a^2)\Big)\right.\\
 &\qquad\left.+\gamma b\Big((d+2)\tau+b^2(3-\alpha(d-1)-\alpha b^2)\Big)\right).
\end{align*}
Solving for $\gamma$, we see the above is positive if and only if
\[
 \gamma\geq\frac{a}{b}\left(\frac{a^2(3-\alpha(d-1)+\alpha a^2)-(d+2)\tau}{b^2(3-\alpha(d-1)-\alpha b^2)+(d+2)\tau}\right)=:\gstar.
\]
As in the higher-dimension case, we will find a rational lower bound on $\gamma$ so that we can prove nonnegativity of a polynomial rather than a transcendental quantity involving Bessel functions. As before, we can bound $\gamma$ below by setting $\alpha=1/(d-1)$, and then apply Lemma~\ref{bblem}. This yields the bound
\[
 \gamma\geq\frac{a^2j_2'(a)}{b^2i_2'(b)}=\frac{a^3((d+4)-a^2)}{b^3\left((d+4)+k(\bmidsq) b^2\right)},
\]
where the constant $k(M)$ is given in Lemma~\ref{bblem} as
\[
 k(M)=\frac{1}{2}+\frac{2}{M}\left(e^{M/4}-1\right).
\]
For $d=2,3$, our bounds on $\bmidsq$ both give us $k(\bmidsq)<1.3$. Then $\gamma-\gstar$ is nonnegative whenever
\begin{align*}
 &a^2(d+4-a^2)\Big(b^2(3-\alpha(d-1)-\alpha b^2)+(d+2)\tau\Big)\\
 &\qquad+b^2(d+4+1.3b^2)\Big(a^2(3-\alpha(d-1)+\alpha a^2)-(d+2)\tau\Big)\geq0.
\end{align*}
As in previous proofs, we write $x=a^2$ and $b^2=x+\tau$, and so the polynomial above can be written as
\begin{align*}
 p(x&,\tau)=\frac{13(d+2)}{10}\tau^3\\
 &\qquad+\frac{1}{10}\Big(-3\alpha x^2+(13(2d+1)+\alpha(3d-53))x+10(d+2)(d+4)\Big)\tau^2\\
 &\qquad+\frac{x}{10}\Big(-6\alpha x^2-3(34-d+2\alpha(26-d))x+20(d+2)(d+4)\Big)\tau\\
 &\qquad-\frac{x^3}{10}\Big(3x\alpha+69+103\alpha-3d\alpha\Big).
\end{align*}
The coefficient of $\tau^3$ is positive by inspection. We will show the coefficients of $\tau$ and $\tau^2$ are also positive for our values of $d$, $\alpha$, and $x$. Note that we're assuming both $\tau\leq\tmid$ and $\tau\leq\tmax$, so we also have $x\leq\xmax$ and $x\leq\xmid$. For $d=3$, this means $x\leq \xmax=45/14<3.22$; for $d=2$, this means $x\leq\xmid<1.852$.

Let's look at the coefficient of $\tau^2$:
\[
 -3\alpha x^2+(13(2d+1)+\alpha(3d-53))x+10(d+2)(d+4).
\]
When $d=3$, we have $\alpha\leq1/2$ and so this becomes
\[
 -3\alpha x^2+(91-44\alpha)x+350\geq -\frac{3}{2}x^2+69x+350.
\]
The right-hand side has roots at $x\approx-4.6, 50.6$ and so is positive for the values of $x$ under consideration, that is, $0\leq x\leq 3.22$.

When $d=2$, we assume $\alpha\leq 51/97$ and so the coefficient of $\tau^2$ becomes
\[
 -3\alpha x^2+(65-47\alpha)x+240\geq -\frac{153}{97}x^2+\frac{3908}{97}x+240.
\]
The right-hand side has roots at $x\approx -4.98, 30.5$ and so is positive for the values of $x$ under consideration, $0\leq x\leq 1.852$.

Finally, we consider the coefficient of $\tau$:
\[
 -6\alpha x^2-3(34-d+2\alpha(26-d))x+20(d+2)(d+4).
\]
When $d=3$, this becomes
\[
 -6\alpha x^2-3(31+46\alpha)x+700\geq -3x^2-162x+700.
\]
The right-hand side quadratic has roots at $x\approx-58.02, 4.02$ and so is positive for the values of $x$ under our consideration.

When $d=2$, our coefficient becomes
\[
 -3\alpha x^2-24(2+3\alpha)x+240\geq-\frac{153}{97}x^2-\frac{8328}{97}x+240.
\]
The right-hand side quadratic has roots at $x\approx-57.1, 2.67$ and so is positive for the values of $x$ under our consideration, $0\leq x\leq 1.852$.

Thus we have shown that $p$ is increasing in $\tau$ for both $d=2,3$ for the values of $x$, $\tau$, and $\alpha$ under consideration, and so $p$ is minimal when $\tau$ is. From Lemma~\ref{atbounds} we have $\tau>x^2/(d+2-x)=\tmin$.

When $d=3$, we have
\begin{align*}
 p\left(x,\frac{x^2}{5-x}\right)&=\frac{5x^3}{2(5-x)^3}m_3(x)\qquad\text{where}\\
 m_3(x)&:=-(5-x)(9407x)\alpha-2x^3+28x^2-96x+355.
\end{align*}
Recall that we wish to show $p$ is nonnegative; it suffices to show $m_3(x)$ is nonnegative. Note that $m_2$ is decreasing in $\alpha$, so taking $\alpha=1/2$ we can make the following lower-bound estimate:
\begin{align*}
 m_3(x)&\geq -2x^3+\frac{49}{2}x^2-\frac{63}{2}x+120\\
 &\geq -2x^2\xmax+\frac{49}{2}x^2-\frac{63}{2}\xmax+120\\
 &\geq18.05x^2+18.56,
\end{align*}
which is positive for all $x$. Thus $m_3(x)$ and hence $p(x,\tmin)$ are both positive for all $x$ and $\alpha$ under consideration.

When $d=2$, we instead have
\begin{align*}
 p\left(x,\frac{x^2}{4-x}\right)&=\frac{4x^3}{5(4-x)^3}m_2(x)\qquad\text{where}\\
 m_2(x)&:=-(4-x)(194-19x)\alpha-5x^3+55x^2-138x+408.
\end{align*}
Note that $m_2$ is decreasing in $\alpha$, since we have $0\leq x\leq\xmid<1.86$. So $m_2$ is minimized when we take $\alpha=51/97$, giving us
\[
 m_2(x)\geq\frac{x}{97}\Big(-485x^2+4366x+384\Big)
\]
The roots of this right-hand side cubic are real and occur at $x=0$ and $x\approx-0.09,9.09$. Since the coefficient of $x^3$ is negative, this means $m_2(x)$ is nonnegative for $0\leq x\leq1.86$ as desired.

This completes our proof, as we have shown that $p$ is nonnegative for all $\tau$, $d$, $\alpha$, and $x$ under consideration, and hence $\gamma-\gstar$ is nonnegative, which was sufficient to establish positivity of $l(r)$.
\end{proof}

\section{Completing the proof}\label{proofsec}
Now that we have established the desired monotonicity of our quotient, we need two more lemmas before we can prove the isoperimetric inequality for the free plate under tension. Our first of these is a special case of more general rearrangement inequalities:
\begin{lemma}\label{monint}\cite[Lemma 14]{inequalitypaper}
For any radial function function $F(r)$ that satisfies the partial monotonicity condition \eqref{moncond} for $\Ostar$,
\[
\int_\Omega F\,dx \leq \int_{\Ostar}F\,dx
\]
with equality if and only if $\Omega=\Ostar$.
For any strictly increasing radial function $F(r)$,
\[
\int_\Omega F\,dx \geq \int_{\Ostar}F\,dx
\]
with equality if and only $\Omega=\Ostar$.
\end{lemma}

The final lemma describes how the eigenvalues change with the dilation of the region, and is used in the proof of the theorem to show we need only consider $\Omega$ with volume equal to that of the unit ball. We will use the notation $s\Omega := \{x\in\RR^d:x/s\in\Omega\}$ for $s>0$.
\begin{lemma}\label{scaling} (Scaling) For all $s>0$, we have
\[
 \omega(\tau,\sigma, \Omega)=s^{4}\omega(s^{-2}\tau,\sigma, s\Omega).
\]
\end{lemma}
The proof is straightforward and nearly identical to that of \cite[Lemma 15]{inequalitypaper}, and so not repeated here.

We can now prove our main result.
\begin{proof}[Proof of Theorem~\ref{mainthm}]
Once we have established inequality \eqref{mainineq} for all regions $\Omega$ of volume equal to that of the unit ball and all $\tau>0$, we obtain \eqref{mainineq} for regions of arbitrary volume, since
\[
\omega(\tau,\sigma, \Omega)=s^{4}\omega(s^{-2}\tau,\sigma,s\Omega)\leq s^{4}\omega(s^{-2}\tau,\sigma, s\Ostar)=\omega(\tau,\Ostar),
\]
for all $s>0$ by Lemma~\ref{scaling}.

Thus it suffices to prove the theorem for $\Omega$ with volume equal to that of the unit ball, so that $\Ostar$ is the unit ball. We may also translate $\Omega$ as in Lemma~\ref{trialfcn}, which leaves the fundamental tone unchanged. Then,
\begin{align*}
\omega &\leq \frac{\int_\Omega N[\rho]\,dx}{\int_\Omega \rho^2\,dx} &&\text{by Lemma~\ref{lemmaboundRC}}\\
&\leq \frac{\int_\Ostar N[\rho]\,dx}{\int_\Ostar \rho^2\,dx}&&\text{by Lemmas~\ref{mondenom}, \ref{monnum}, and~\ref{monint}}\\
&=\omega^*,
\end{align*}
by applying the equality condition in Lemma~\ref{lemmaboundRC}. Finally, if equality holds, then $\Omega$ must be a ball, by the equality statement in Lemma~\ref{monint}.
\end{proof}

\section*{Acknowledgments} 
This research was partially supported by the University of Minnesota's Faculty Development Single Semester Leave. I would also like to thank Richard Laugesen for tirelessly providing advice and wisdom on matters both mathematical and professional.

\end{document}